\documentclass[11pt, letterpaper]{article}

\usepackage{weak_moment}
\usepackage{verbatim}

\newcommand{\optx}{x^*}
\newcommand{\optd}{d^*}
\newcommand{\xt}[1]{x_{#1}}
\newcommand{\dt}[1]{d_{#1}}
\newcommand{\gt}[1]{g_{#1}}
\newcommand{\nbu}{\cbu \log 1 / \delta}
\newcommand{\nit}{\cgd \log dn}

\newcommand{\rad}{10^6\lprp{\sqrt{\frac{d}{n}} + \lprp{\frac{d}{n}}^{\frac{\alpha}{1 + \alpha}} + \lprp{\frac{\log 1 / \delta}{n}}^{\frac{\alpha}{1 + \alpha}}}}
\newcommand{\radthm}{10^8\lprp{\sqrt{\frac{d}{n}} + \lprp{\frac{d}{n}}^{\frac{\alpha}{1 + \alpha}} + \lprp{\frac{\log 1 / \delta}{n}}^{\frac{\alpha}{1 + \alpha}}}}
\newcommand{\smean}{\tld{\mu}}

\newcommand{\mt}{MT}
\newcommand{\thr}{\tau}

\newcommand{\vthr}{100 n^{\frac{1}{1 + \alpha}}d^{-\frac{(1 - \alpha)}{2(1 + \alpha)}}}
\newcommand{\dd}{\Delta}

\newcommand{\radAlg}{20\radst}
\newcommand{\radAlgG}{25\radst}
\newcommand{\radRet}{30\radst}
\newcommand{\radst}{r^*}

\newcommand{\proj}{\mathcal{P}}
\newcommand{\projp}{\mathcal{P}^\perp}

\newcommand{\tld}[1]{\widetilde{#1}}
\newcommand{\ts}{\thicksim}
\newcommand{\normtto}[1]{\norm{#1}_{2 \to 1}}
\newcommand{\ns}{\mc{N}}

\newcommand{\rob}{\eta}

\newcommand{\cgd}{10^6}
\newcommand{\cbu}{4000}

\def\authornotes{1pt}

\ifdim\authornotes=1pt
\newcommand{\Ynote}[1]{\footnote{\color{ForestGreen}Yeshwanth: #1}}
\newcommand{\Nnote}[1]{\footnote{\color{Orange}Nilesh: #1}}
\else
\newcommand{\Ynote}[1]{}
\newcommand{\Nnote}[1]{}
\fi

\begin{document}
\title{Optimal Mean Estimation without a Variance}
\author{Yeshwanth Cherapanamjeri\thanks{Department of Electrical Engineering and Computer Science. \texttt{yeshwanth@berkeley.edu}} \and Nilesh Tripuraneni\thanks{Department of Electrical Engineering and Computer Science. \texttt{nilesh\_tripuraneni@berkeley.edu}} \and Peter L. Bartlett\thanks{Department of Electrical Engineering and Computer Science, Department of Statistics. \texttt{peter@berkeley.edu}} \and Michael I. Jordan\thanks{Department of Electrical Engineering and Computer Science, Department of Statistics. \texttt{jordan@cs.berkeley.edu}}}

\date{}

\maketitle




\begin{abstract}
    We study the problem of heavy-tailed mean estimation in settings where the variance of the data-generating distribution does not exist. Concretely, given a sample $\bm{X} = \{X_i\}_{i = 1}^n$ from a distribution $\mc{D}$ over $\mb{R}^d$ with mean $\mu$ which satisfies the following \emph{weak-moment} assumption for some ${\alpha \in [0, 1]}$:
    \begin{equation*}
        \forall \norm{v} = 1: \mb{E}_{X \ts \mc{D}}[\abs{\inp{X - \mu}{v}}^{1 + \alpha}] \leq 1,
    \end{equation*}
    and given a target failure probability, $\delta$, our goal is to design an estimator which attains the smallest possible confidence interval as a function of $n,d,\delta$. For the specific case of $\alpha = 1$, foundational work of Lugosi and Mendelson exhibits an estimator achieving subgaussian confidence intervals, and subsequent work has led to computationally efficient versions of this estimator. Here, we study the case of general $\alpha$, and establish the following information-theoretic lower bound on the optimal attainable confidence interval:
    \begin{equation*}
        \Omega \lprp{\sqrt{\frac{d}{n}} + \lprp{\frac{d}{n}}^{\frac{\alpha}{(1 + \alpha)}} + \lprp{\frac{\log 1 / \delta}{n}}^{\frac{\alpha}{(1 + \alpha)}}}.
    \end{equation*}
    Moreover, we devise a computationally-efficient estimator which achieves this lower bound. 
\end{abstract}

\thispagestyle{empty}
\setcounter{page}{0}
\newpage

\section{Introduction}

We aim to solve a fundamental problem in statistical inference---mean estimation---under minimal assumptions. Formally, we seek the tightest confidence interval (up to constants) achievable for the problem of mean estimation, equipped solely with a weak moment assumption on the $X_i$ (say, when $X_i$ are drawn from a multivariate t-distribution):
\begin{problem}
    Consider a sequence of $n$ i.i.d. vectors from a distribution $\mc{D}$ over $\mb{R}^d$ with mean $\mu$ satisfying the following weak moment condition for some $0 \leq \alpha \leq 1$: 
    \begin{equation*}
        \forall v \in \mb{R}^d, \norm{v} = 1:\ \mb{E} \lsrs{\abs{\inp{v}{X_i - \mu}}^{1 + \alpha}} \leq 1, \quad \forall i \in \{1, \hdots, n \}. \tag{MC}
    \end{equation*}
    Given a confidence level $\delta > 0$, output an estimate $\hat{\mu}$ with the smallest radius $r_\delta$ satisfying:
    \begin{equation*}
        \mb{P} \lbrb{\norm{\hat{\mu} - \mu} > r_\delta} \leq \delta.
    \end{equation*}
\end{problem}
For the case of $\alpha = 1$, this amounts to an assumption on the spectral norm of the covariance matrix of the distribution $\mc{D}$. Even in this special case, estimators with optimal confidence interval, $r_\delta$, were only recently discovered (\cite{lugosi2017sub,catoni2017dimension}), building upon the one-dimensional median-of-means (MoM) framework introduced in \cite{NemYud83}. These estimators achieve the following rate:
\begin{equation*}
    r_\delta = O\lprp{\sqrt{\frac{d}{n}} + \sqrt{\frac{\log 1 / \delta}{n}}}.
\end{equation*}

Unfortunately, however, there is no known polynomial-time algorithm to compute the estimators proposed in these papers. Computationally-efficient estimators achieving the optimal confidence interval were first proposed in \cite{hopkins2018sub} based on the sum-of-squares family of semidefinite relaxations of the estimator from \cite{lugosi2017sub}. By combining these ideas with a algorithm based on gradient descent, faster mean estimators were subsequently developed in \cite{cfb}. Perhaps surprisingly, this line of work shows as long as the variance of the random vector exists, neither statistical nor computational efficiency is necessarily sacrificed when estimating $\mu$.  In particular, the dependence on $d, n, \delta$ of the confidence interval for MoM estimators, when the samples have bounded second moments, exactly matches the optimal dependence on $d, n, \delta$, \textit{when the samples $X_i$ are Gaussian}.


When $\alpha<1$, the situation is markedly different. In the one-dimensional case, the (optimal) achievable
radius satisfies \cite{devroye2016sub}:
\begin{equation*}
    r_\delta = O\lprp{\frac{\log 1 / \delta}{n}}^{\frac{\alpha}{1 + \alpha}},
\end{equation*} 
which can be achieved by a univariate MoM-style estimator. Even in one dimension, the lack of a second moment degrades the information-theoretic bound with respect to both $n$ and $\delta$. Unlike the case $\alpha \geq 1$, Gaussian-like confidence intervals cannot be obtained in this regime. Moreover, in $d$ dimensions, there is very little known about the optimal achievable radius save for the fact that one can obtain the following trivial rate by applying the univariate estimator coordinate-wise:\footnote{This follows by the triangle inequality and union bound.}
\begin{equation*}
    r_\delta = \tld{O} \lprp{\sqrt{d} \lprp{\frac{\log 1 / \delta}{n}}^{\frac{\alpha}{1 + \alpha}}}.
\end{equation*}

Two natural questions thus present themselves. First, in the regime where $\alpha < 1$, what is minimax-optimal rate for mean estimation in higher dimensions? Second, can this (hitherto unknown) rate also be achieved in polynomial time?

The primary contribution of the current paper is to present to sharp answers to both of these questions. These answers are contained in the following two theorems, the first of which presents a rate that is achievable by a polynomial-time algorithm and the second of which establishes the optimality of this rate.
\begin{theorem}
    \label{thm:ub_simple}
    Let $\bm{X} = X_1, \dots, X_n$ be iid random vectors with mean $\mu$, satisfying the weak moment assumption (MC) for some known $\alpha > 0$. There is a polynomial-time algorithm which, when given inputs $\bm{X}$ and a target confidence $\delta > 2^{-\frac{n}{16000}}$, returns a point $x^*$ satisfying:
    \begin{equation*}
        \norm{x^* - \mu} \leq  \radthm,
    \end{equation*}
    with probability at least $1 - \delta$.
\end{theorem}
\cref{thm:ub_simple} is based on a two-stage estimation procedure. In the first step, the set of inputs $\bm{X}$ is truncated to discard samples that are too far from the empirical centroid of $\bm{X}$. The second step uses an estimation-to-testing framework for heavy-tailed estimation \citep{cfb} by first setting up a testing problem which decides if a candidate mean is close to the true mean $\mu$ and, subsequently, using this procedure to improve the estimate. By iterating this procedure, we eventually converge to a good approximation of the true mean.

Our second main result is a matching minimax lower bound establishing the optimality of the rate in \cref{thm:ub_simple}.
\begin{theorem}
    \label{thm:lb_simple}
    Let $n > 0$ and let $\delta \in \lprp{e^{-\frac{n}{4}}, \frac{1}{4}}$. Then there exists a set of distributions $\mc{F}$ over $\mb{R}^d$ such that each $\mc{D} \in \mc{F}$ obeys the weak moment condition (MC) for some $\alpha > 0$, and any estimator $\hat{\mu}$ satisfies:
    \begin{equation*}
        \mb{P}_{\mc{D} \in \mc{F}} \lbrb{\norm{\hat{\mu} (\bm{X}) - \mu(\mc{D})} \geq \frac{1}{24}\cdot\max \lprp{\lprp{\frac{d}{n}}^{\frac{\alpha}{1 + \alpha}}, \sqrt{\frac{d}{n}}, \lprp{\frac{\log 2 / \delta}{n}}^{\frac{\alpha}{1 + \alpha}}}} \geq \delta,
    \end{equation*}
    where the data $\bm{X}$ are generated iid from $\mc{D}$. 
\end{theorem}
The main challenge in proving \cref{thm:lb_simple} lies in obtaining tight dependence on dimension. The proof begins by using a standard reduction from estimation to testing for proving minimax rates \citep[see, for example,][Chapter 15]{wainwright}. We then avoid traditional Fano-style information-theoretic approaches, however, in establishing difficulty of the testing problem. We instead take a Bayesian approach and use a carefully chosen set of discrete distributions to instantiate the testing problem, allowing us to establish a sharp dependence on dimension in our lower bound, once various technical challenges are surmounted.
 

Together, \cref{thm:ub_simple} and \cref{thm:lb_simple} have the following implications for the problem of mean estimation without a variance:
\begin{itemize}
    \item In the case in which the failure probability $\delta$ is a constant, our upper and lower bounds simplify to $O(\sqrt{d / n} + (d / n)^{\alpha / (1 + \alpha)})$. Interestingly, \cref{thm:ub_simple} and \cref{thm:lb_simple} reveal the existence of a phase transition in the estimation rate when $n \asymp d$---the estimation rate is dominated by $\sqrt{\frac{d}{n}}$ when $n \lesssim d$ and $(\frac{d}{n})^{\alpha/(1+\alpha)}$ when $n \gtrsim d$.
    \item While it is established in \cite{devroye2016sub} that it is impossible to obtain subgaussian rates in this setting even in one dimension, our results reveal a decoupling between the terms depending on the failure probability and the dimension that parallels the finite-variance case (where $\alpha = 1$).
    \item Finally, our results also extend to the more general problem of mean estimation under adversarial corruption, which has received much recent attention in both the computer science and statistics communities. In this setting, an adversary is allowed to inspect the data points and \emph{arbitrarily} corrupt a fraction $\eta$ of them. We recover the mean up to an error of $O(\eta^{\alpha / (1 + \alpha)})$ which is information-theoretically optimal (\cref{thm:rob_lb}). Furthermore, as a consequence of \cref{thm:lb_simple}, our sample complexity of $O(d / \eta)$ is also optimal.
\end{itemize}

The main technical challenge in establishing our upper bound is the analysis of the  estimation-to-testing framework of \cite{cfb} in the weak-moment setting. The analysis in \cite{cfb} makes critical use of the decomposition of the variance of sums of independent random variables which does not hold in our setting. This allows tight control of the \emph{second} moments of $\sum_{i = 1}^m X_i$ and $\norm{X - \mu}$, which are crucial to that analysis. Despite the lack of such decompositions for weak moments, we establish tight control over the appropriate quantities allowing us to establish our optimal recovery guarantees.

Similarly, the presence of weak moments also complicates the task of establishing a matching lower bound with tight dependence on the dimension $d$. The main difficulty is in proving the optimality of the dimension-dependent term, $(d / n)^{\alpha / (1 + \alpha)}$. For the specific case where $\alpha = 1$, the lower bound may be proved within the estimation-to-testing framework by utilizing a distribution over a well-separated collection of Gaussian distributions. However, this approach fails for the weak-moment mean estimation problem; indeed, hypercontractivity properties of Gaussian distributions ensure a bounded variance leading to a lower bound that scales as $1 / \sqrt{n}$ as opposed to the slower rate $n^{- \alpha / (1 + \alpha)}$. To prove our lower bound, we instead use a collection of carefully chosen distributions with discrete supports whose means are separated by $O((d / n)^{\alpha / (1 + \alpha)})$. Further challenges arise at this point---if we follow the standard path of bounding the complexity of the testing problem in terms of pairwise $f$-divergences between distributions in the hypothesis set, we obtain vacuous bounds. We instead directly analyze the posterior distribution obtained from the framework and show that random independent samples from the posterior tend to be well separated, yielding our tight lower bound.
\section{Related Work}
\label{sec:rel_work}

There has been much interest in designing information-theoretically optimal estimators for fundamental inferential tasks under minimal assumptions on the distributions generating the data \cite{NemYud83,jerrum1986random,alon1999space,lugosi2017sub,lugosi2018risk,hopkins2018sub,cfb,chkrt,trimmed,deplecmean,fredhtmean,depersin2020spectral,depersin2020robust,htsurvey}. In the one-dimensional setting, estimators achieving the information theoretically-optimal \emph{subgaussian} rate were obtained in the seminal work of \cite{alon1999space,jerrum1986random,NemYud83}. In recent years, focus has shifted towards the high-dimensional setting where one aims for optimal recovery error in terms of the number of samples $n$, the dimension $d$, and the failure probability $\delta$, without making strong distributional assumptions such as Gaussianity. As a consequence of this effort, information-theoretically optimal estimators have been developed for mean estimation \cite{lugosi2017sub}, linear regression \cite{lugosi2018risk} and covariance estimation \cite{mzcovariance}. However, the estimators proposed in these works lack computationally efficient algorithms to compute them. The first computationally efficient estimator was proposed by \cite{hopkins2018sub} and its runtime and analysis were subsequently improved in \cite{cfb,deplecmean,fredhtmean}. Subsequently, improved algorithms have been devised for linear regression and covariance estimation \cite{chkrt}. Most recently, \cite{depersin2020spectral} extended the approach of \cite{fredhtmean} to linear regression, improving on \cite{chkrt} in settings where the covariance matrix of the data-generating distribution is known. We direct the interested reader to the survey \cite{htsurvey} and the references therein for more detailed discussion of this line of research. Note that the optimal recovery guarantee obtainable in all these settings is the subgaussian rate. This is provably not possible in the weak-moment scenario even in the one-dimensional case as evidenced by our lower bound.

Another approach towards achieving distributional robustness which has received much attention in the computer science community is robust estimation under a contamination model. In this setting, an adversary is allowed to inspect a set of data points generated from a well-behaved distribution and can arbitrarily corrupt a fraction $\eta$ of them according to their choosing. Broadly speaking, the primary goal of this field is to obtain optimal recovery of the underlying parameter as a function of the corruption factor $\eta$ as opposed to achieving optimal dependence on $n, d, \delta$. Starting with the foundational works of \cite{huber64,Tuk60}, which obtain information-theoretically optimal (albeit computationally intractable) estimators, computationally efficient estimators are now known for a range of statistical estimation problems in various settings \cite{lai2016agnostic,diakonikolas2016robust,charikar2017learning,steinhardt2017resilience,diakonikolas2017being,diakonikolas2018list,kothari2018robust,diakonikolas2018robustly,DBLP:conf/nips/DongH019,cheng2019high}. Since the literature of this field is vast, we restrict ourselves to the specific setting of mean estimation and direct the reader to \cite{diakonikolas2019recent} for more context on these developments. For the mean estimation problem under adversarial corruption, this line of work has resulted in estimators which succeed with constant probability, say $2/3$, and achieve the optimal recovery error of $\sqrt{\eta}$ assuming the data is drawn from a distribution with finite covariance \cite{DBLP:conf/nips/DongH019,cheng2019high,deplecmean}. In addition, \cite{trimmed,deplecmean} achieve this rate along with the optimal dependence on $\delta$ in the recovery guarantees. A corollary of our work extends these results to the setting where the covariance matrix is not defined. 
\section{Algorithm}
\label{sec:alg}
In this section, we describe our algorithm for mean estimation problem in the setting of the weak moment condition (MC). We build on the approach of \cite{cfb} which operates in the setting where the covariance matrix of the distribution generating the data is defined. However, the absence of second moments complicates the design of our algorithm leading to a more intricate procedure. Concretely, our algorithm is comprised of the following three broad stages:
\begin{enumerate}
    \item Data Pruning: First, we compute an initial crude estimate of the mean which is within $O(\sqrt{d})$ of $\mu$ and then proceed to use this estimate to filter out data points which are far from our estimate. This truncation must be chosen carefully to ensure that the mean of the truncated data points still approximates the true mean well. As a consequence, this truncation is a function of $n$, $d$ and $\alpha$. \cref{alg:inite,alg:part} describe this crude estimation procedure and truncation step in greater detail. Intriguingly, this additional thresholding step (which is not necessary in the $\alpha = 1$ case) is critical to achieve the statistically optimal confidence intervals in this scenario.
    \item Data Batching: In this stage, the data points that survive the truncation procedure in \cref{alg:part} are then divided into $k$ bins and mean estimates are computed by averaging the set of points in each bin. The number of bins is chosen depending on the desired failure probability, $\delta$. The precise setting of parameters is described in \cref{alg:loce}.
    \item Median Computation: Finally, the bucket estimates, $Z_i$, produced in the previous stage are aggregated to produce our final estimate of the mean. The procedure to do this follows along the testing-to-estimation framework for robust estimation explored in \cite{cfb}. Here, one first designs a procedure to \emph{test} whether a given candidate, $x$, is close to $\mu$. Subsequently, a solution to the testing program is then used to improve the estimate. In this setting, one shows that the testing program can be used to estimate both $\norm{x - \mu}$ and an approximation to $\Delta = \frac{\mu - x}{\norm{x - \mu}}$ which may be used to improve our estimate. \cref{alg:deste,alg:geste} and \cref{alg:gd} display the estimation and tuning components of this stage. The testing program is defined in \ref{eq:mt} and is discussed in more detail subsequently.
\end{enumerate}

\begin{algorithm}[H]
\caption{Mean Estimation}
\label{alg:meste}
\begin{algorithmic}[1]
    \STATE \textbf{Input}: Data Points $\bm{X} \in \mb{R}^{n\times d}$, Target Confidence $\delta$
    \STATE $x^\dagger \gets \text{Initial Mean Estimate}(\{X_1, \dots, X_{n / 2}\})$
    \STATE $\bm{Z} \leftarrow \text{Produce Bucket Estimates}(\lbrb{X_{n / 2 + 1}, \dots, X_n}, x^\dagger, \delta)$
    \STATE $T \gets \nit$
    \STATE $\optx = \text{Gradient Descent}(\bm{Z}, x^\dagger, T)$
    \STATE \textbf{Return: } $\optx$
\end{algorithmic}
\end{algorithm}
\vspace{-.5cm}
\begin{algorithm}[H]
\caption{Gradient Descent}
\label{alg:gd}
\begin{algorithmic}[1]
    \STATE \textbf{Input}: Bucket Means $\bm{Z} \in \mb{R}^{k\times d}$, Initialization $x^\dagger$, Number of Iterations $T$
    
    \STATE $\optx, \xt{0} \leftarrow x^\dagger$ and  $\optd, \dt{0} \leftarrow \infty$

    \FOR{$t = 0:T$}
    \STATE $\dt{t} \leftarrow \text{Distance Estimation}(\bm{Z}, \xt{t})$
    \STATE $\gt{t} \leftarrow \text{Gradient Estimation}(\bm{Z}, \xt{t})$

    \IF {$\dt{t} < \optd$}
        \STATE $\optx \leftarrow \xt{t}$
        \STATE $\optd \leftarrow \dt{t}$
    \ENDIF

    \STATE $\xt{t + 1} \leftarrow \xt{t} + \frac{1}{20} \dt{t}\gt{t}$
    \ENDFOR
    \STATE \textbf{Return: } $\optx$
\end{algorithmic}
\end{algorithm}
\vspace{-.5cm}
\begin{algorithm}[H]
    \caption{Produce Bucket Estimates}
    \label{alg:loce}
    \begin{algorithmic}[1]
        \STATE \textbf{Input}: Data Points $\bm{X} \in \mb{R}^{n\times d}$, Mean Estimate $x^\dagger$, Target Confidence $\delta$
    
        \STATE $\bm{Y} \gets \text{Prune Data} (\bm{X}, x^\dagger)$
        \STATE $m \gets \abs{\bm{Y}}$
        \STATE $k \leftarrow \nbu$
        \STATE Split data points into $k$ buckets with bucket $\mathcal{B}_i$ consisting of the points $X_{(i - 1)\frac{m}{k} + 1}, \dots, X_{i\frac{m}{k}}$
        \STATE $Z_i \leftarrow \mathrm{Mean}(\mathcal{B}_i)\ \forall\ i \in [k]$ and  $\bm{Z} \leftarrow (Z_1, \dots, Z_k)$
        \STATE \textbf{Return: } $\bm{Z}$ 
    \end{algorithmic}
\end{algorithm}
\vspace{-.5cm}
\begin{algorithm}[H]
    \caption{Initial Mean Estimate}
    \label{alg:inite}
    \begin{algorithmic}[1]
        \STATE \textbf{Input: } Set of data points $\bm{X} = \{X_i\}_{i = 1}^n$
        \STATE $\hat{\mu} \gets \argmin_{X_i \in \bm{X}} \min \lbrb{r > 0: \sum_{j = 1}^n \bm{1} \lbrb{\norm{X_j - X_i} \leq r} \geq 0.6n}$ \label{line:min_init}
        \STATE \textbf{Return: } $\hat{\mu}$
    \end{algorithmic}
\end{algorithm}
\vspace{-.5cm}
\begin{algorithm}[H]
    \caption{Prune Data}
    \label{alg:part}
    \begin{algorithmic}[1]
        \STATE \textbf{Input: } Set of data points $\bm{X} = \{X_i\}_{i = 1}^n$, Mean Estimate $x^\dagger$
        \STATE $\tau \gets \max\lprp{\vthr, 100 \sqrt{d}}$
        \STATE $\mc{C} \gets \{X_i: \norm{X_i - x^\dagger} \leq \tau\}$
        \STATE \textbf{Return: } $\mc{C}$
    \end{algorithmic}
\end{algorithm}
\vspace{-0.5cm}
\noindent
\begin{minipage}[H]{.49\textwidth}
\begin{algorithm}[H]
\caption{Distance Estimation}
\label{alg:deste}
\begin{algorithmic}[1]
    \STATE \textbf{Input}: Data Points $\bm{Z} \in \mb{R}^{k\times d}$, Current point $x$
    \STATE $d^* = \argmax_{r > 0} \ref{eq:mt}(x,r,\bm{Z}) \geq 0.9 k$
    \STATE \textbf{Return: } $d^*$\newline
\end{algorithmic}
\end{algorithm}
\vspace{0.93cm}
\end{minipage}
\hfill \vspace{-.5cm}
\begin{minipage}[H]{.49\textwidth}
\begin{algorithm}[H]
\caption{Gradient Estimation}
\label{alg:geste}
\begin{algorithmic}[1]
    \STATE \textbf{Input}: Data Points $\bm{Z} \in \mb{R}^{k\times d}$, Current point $x$

    \STATE $d^*$ = Distance Estimation$(\bm{Z}, x)$
    \STATE $(v, X) = \ref{eq:mt}(x,d^*,\bm{Z})$
    \STATE $g \gets \text{Top Singular Vector} (X_v)$
    \STATE \textbf{Return: }$g$
\end{algorithmic}
\end{algorithm}
\vspace{.5cm}
\end{minipage}

As in \cite{hopkins2018sub,cfb}, the following polynomial optimization problem and its semidefinite relaxation play a key role in our subsequent analysis. Intuitively, given a test point, $x$, the program searches for a direction (denoted by $v$) such that a large fraction of the bucket estimates, $Z_i$, are far away from $x$ along $v$. Formally, the polynomial optimization problem, parameterized by $x$, $r$ and $\bm{Z}$, is defined below:
\begin{gather*}
    \max \sum_{i = 1}^k b_i \\
    \text{s.t } b_i^2 = b_i \\
    \norm{v}^2 = 1 \\
    b_i (\inp{v}{Z_i - x} - r) \geq 0 \ \forall i \in [k] \tag{\textbf{MTE}} \label{eq:mte}.
\end{gather*}
The binary variables $b_i$ indicates whether the $i^{th}$ bucket mean $Z_i$ is far away along $v$. Unfortunately, the binary constraints on $b_i$, the restriction of $v$ and the final constraint make this problem nonconvex and there are no efficient algorithms known to compute it. Accordingly, we work with the semidefinite relaxation defined as follows:
\begin{gather*}
    \max \sum_{i = 1}^k X_{1, b_i} \\
    X_{1, b_i} = X_{b_i, b_i} \\
    \sum_{j = 1}^d X_{v_j, v_j} = 1 \\
    \inp{v_{b_i}}{Z_i - x} \geq X_{b_i,b_i}r \ \forall i \in [k]\\
    X_{1,1} = 1\\ 
    X \succcurlyeq 0, \tag{\textbf{MT}} \label{eq:mt}
\end{gather*}
where $v_{b_i} = [X_{b_i, v_1}, \dots, X_{b_i, v_d}]^\top$. The matrix $X \in \psd[(k + d + 1)]$ is symbolically indexed by $1$ and the variables $b_1,\ldots,b_k$ and $v_1,\ldots,v_d$.
We will restrict ourselves to analyzing \ref{eq:mt} and will refer to the program initialized with $x$, $r$ and $\bm{Z}$ as \ref{eq:mt}$(x, r, \bm{Z})$. We will use $(v, X) = \ref{eq:mt}(x, r, \bm{Z})$ to denote the optimal value, $v$, and solution, $X$, of \ref{eq:mt} initialized with $x$, $r$ and $\bm{Z}$. For the sake of clarity, $\ref{eq:mt}(x, r, \bm{Z})$ in the absence of any specified output will refer to the optimal value of $\ref{eq:mt}(x, r, \bm{Z})$. 
\section{Proof Overview}
\label{sec:overview}

In this section, we outline the main steps in proving \cref{thm:ub_simple}, providing full details in the Appendix. For ease of exposition, we restrict most of our attention to those steps of our proof which are complicated by the weaker assumptions used in our work. From \cref{sec:alg}, we see that our estimation procedure is divided into three stages:
\begin{enumerate}
    \item We obtain an initial coarse estimate, $\hat{\mu}$, of $\mu$ (\cref{alg:inite});
    \item We then use $\hat{\mu}$ to prune data points far away from $\mu$ (\cref{alg:part});
    \item Finally, the remaining data points are incorporated into a gradient-descent algorithm to obtain our final estimate (\cref{alg:meste}).
\end{enumerate}
To obtain our tight rates, we crucially require the following correctness guarantees on these three steps, each with high probability: our initial estimate $\hat{\mu}$ is within a radius of $O(\sqrt{d})$ of $\mu$, a large fraction of data points pass the pruning steps in \cref{alg:part}, and finally, a tight analysis on the error of the gradient descent procedure in \cref{alg:meste}. The first two steps are novel to the weak-moment setting and the third step, while explored previously for the case $\alpha = 1$, is complicated here due to the lack of strong decomposition structure in the weak moments. 

To deal with these difficulties, we establish two crucial structural lemmas, proved in \cref{sec:aux}, on distributions satisfying weak-moment conditions. The first is a bound on the $1 + \alpha$ moments of the lengths of such random vectors:

\begin{lemma}
    \label{lem:lmom}
    Let $X$ be a zero-mean random vector satisfying the weak-moment assumption for some $0 \leq \alpha \leq 1$. We have the following bound:
    \begin{equation*}
        \mb{E} [\norm{X}^{1 + \alpha}] \leq \frac{\pi}{2}\cdot d^{\frac{1 + \alpha}{2}}.
    \end{equation*}
\end{lemma}

As we will see, this lemma is crucial in all three steps of our analysis. Note that the upper bound obtained by the lemma is tight up to a small constant factor. (A standard Gaussian random vector yields an upper bound of $d^{(1 + \alpha) / \alpha}$). The proof follows by first considering independent random Gaussian projections of the random vector along with an application of Jensen's inequality.

The second key lemma is a bound on the $1 + \alpha$ moments of the sum of random variables satisfying weak-moment assumptions.  This lemma plays a key role in obtaining tight bounds on the accuracy of the gradient-descent procedure in \cref{alg:meste}:

\begin{lemma}
    \label{lem:srv}
    Let $X_1, \dots, X_n$ be $n$ mean-zero i.i.d. random variables satisfying the following bound, for some $0 \leq \alpha \leq 1$:
    \begin{equation}
        \mb{E} [\abs{X_i}^{1 + \alpha}] \leq 1.
    \end{equation}
    We have:
    \begin{equation}
        \mb{E} \lsrs{\abs*{\sum_{i = 1}^n X_i}^{1 + \alpha}} \leq 2n.
    \end{equation}
\end{lemma}

The proof of this 
lemma uses techniques employed previously for establishing the Nemirovskii inequalities for Banach spaces but crucially hold even when the covariance of the matrices are not defined \citep{nemineq}. 
We now sketch the argument establishing guarantees on the first two steps of the algorithm. Firstly, from \cref{lem:lmom}, we have that most of the sample sample data points are within a radius of $O(\sqrt{d})$ of $\mu$ with high probability. Therefore, the value of the minimizer in Line~\ref{line:min_init} of \cref{alg:inite} is $O(\sqrt{d})$ (by simply picking any of the data points close to $\mu$) and, furthermore, at least one of the points within $O(\sqrt{d})$ of $\hat{\mu}$ must be at a distance at most $O(\sqrt{d})$ from $\mu$ as most data points are close to $\mu$. This establishes the required correctness guarantees on \cref{alg:inite} which is formalized in \cref{lem:init}. For the second step, we condition on the success of the first step and note that our threshold, $\tau$, is chosen such that all points within $O(\sqrt{d})$ of $\mu$ are within $\tau$ of $\hat{\mu}$. Another application of \cref{lem:lmom} now ensures that most data points pass this threshold, establishing correctness for the second step of our procedure. This is outlined in \cref{lem:part}. We devote the following subsection to the final and most technical step in our analysis.

\subsection{Gradient Descent Analysis}
\label{ssec:gd_main}
In this section we sketch the main steps in the analysis of the gradient-descent procedure used in \cref{alg:meste}. Throughout this subsection, we assume that the previous two steps of the procedure are successful; that is, $\hat{\mu}$ constructed as part of \cref{alg:part} is within $O(\sqrt{d})$ of $\mu$ and as a consequence at least $\Omega (n)$ points are used to construct the bucket estimates. Now, let $\{Z_i\}_{i = 1}^k$ denote the bucket estimates produced as part of \cref{alg:loce} and let $\tld{\mu} = \mb{E} [Z_i]$. From prior work \cite{cfb}, the estimate returned by \cref{alg:meste} is within a radius of $O(r^*)$ of $\tld{\mu}$, where:
\begin{equation*}
    r^* \coloneqq \min \lbrb{r > 0: \ref{eq:mt} (\tld{\mu}, r, \bm{Z}) \leq 0.05k}.
\end{equation*}
Therefore, the error of our estimate may be upper bounded by the sum of two terms: the first is the degree to which $\tld{\mu}$ approximates $\mu$ and the second is an upper bound on $r^*$. We will see that there is an inherent tradeoff between these two terms---by picking the threshold $\tau$ in \cref{alg:part} to be extremely large, $\tld{\mu}$ may be an arbitrarily good approximation of $\mu$ but our bound on $r^*$ may be poor. We now state a  structural lemma capturing the tradeoff between $r^*$ and $\tau$.
\begin{lemma}
    \label{lem:rst_bound}
    Let $\bm{Z} = \{Z_i\}_{i = 1}^k$ be $k$ iid random vectors with mean $\tld{\mu}$ and covariance matrix $\Lambda$.  Suppose that $\mb{E} \lsrs{\abs{\inp{v}{Z_i - \tld{\mu}}}^{1 + \alpha}} \leq \beta$ for all $\norm{v} = 1$. We have:
    \begin{equation*}
        r^* \leq 1000\lprp{\sqrt{\frac{\Tr{\Lambda}}{k}} + \beta^{1 / (1 + \alpha)}} \text{ where } r^* \coloneqq \min \{r > 0: \ref{eq:mt} (\tld{\mu}, r, \bm{Z}) \leq 0.05k\},
    \end{equation*}
    with probability at least $1 - e^{- k / 800}$.
\end{lemma}
Given this lemma, the main remaining difficulty is in obtaining bounds on $\mb{E}\lsrs{\abs{\inp{Z_i}{v}}^{1 + \alpha}}$, $\Tr(\Lambda)$, and the deviation of $\tld{\mu}$ from $\mu$. Note that from the definition of \cref{alg:loce}, the $Z_i$ are means of truncated data points and hence their covariance matrix is well defined. 

To obtain bounds on $\norm{\tld{\mu} - \mu}$, an application of Markov's inequality and \cref{lem:lmom} establishes that the probability of $X_i$ being truncated in \cref{alg:part} is at most $O(d / n)$. Then, a standard variational argument along with our weak-moment assumption allows us to bound $\norm{\tld{\mu} - \mu}$ (see \cref{lem:mnshft} for more details). 

We obtain a bound on $\Tr{\Lambda}$ by first observing that each $X_i$ used to compute one of the bucket estimates, $Z_j$, is truncated with respect to its distance from $\hat{\mu}$. By an application of the triangle inequality, we infer that all of the $X_i$ used in the computation satisfy $\norm{X_i - \hat{\mu}} \leq \tau + O(\sqrt{d})$. Therefore, to bound $\Tr{\Lambda}$, all we need is a bound on ${\mb{E} [\norm{X_i - \tld{\mu}}^2 \bm{1} \lbrb{\norm{X_i - \hat{\mu}} \leq \tau + O(\sqrt{d})}]}$. Another appeal to \cref{lem:lmom} and a straightforward truncation argument establishes a bound of $O(\sqrt{d / n} + (d / n)^{\alpha / (1 + \alpha)})$ (see \cref{lem:trsig,lem:part}).

For the final term, observe that the $Z_i$ are averages of truncated versions of $X_i$. A simple argument shows that truncated $X_i$ satisfy a weak-moment bound (\cref{lem:mnshft}). Therefore, a direct application of \cref{lem:srv} yields a bound of $O(k / n)$ on $\beta$. Incorporating these bounds into \cref{lem:rst_bound} and the gradient-descent framework for heavy-tailed estimation from \cite{cfb} concludes the proof of \cref{thm:ub_simple}.
\section{Lower Bound}
\label{sec:lb}

In this section, we present a lower bound for heavy-tailed regression which shows that the bound obtained in \cref{thm:ub_simple} is tight. 

For a given dimension $d$, and sample size $n$, we will consider a family of distributions parameterized by size $d / 2$ subsets of $[d]$. That is, we will consider a family of distributions $\mc{F} = \{\mc{D}_S: S \subset [d] \text{ and  } \abs{S} = d / 2\}$. Now, for each particular distribution $\mc{D}_S$, we have $X \ts \mc{D}_S$ as follows:
\begin{equation*}
    X = \begin{cases}
            0, & \text{with probability } 1 - \frac{d}{8n} \\
            n^{\frac{1}{1 + \alpha}}\cdot d^{- \frac{(1 - \alpha)}{2(1 + \alpha)}} \cdot \bm{e_i}, & \text{for $i \in S$ with probability } \frac{1}{4n}.
        \end{cases}
\end{equation*}

We will first show that the distribution $\mc{D}_S$ satisfies the $1 + \alpha$ moment condition. 

\begin{lemma}
    \label{lem:const1pa}
    Let $X \ts \mc{D}_S$ for some $S \subset [d]$ such that $\abs{S} = d / 2$. Then, $X$ satisfies the following:
    \begin{equation*}
        \forall v: \norm{v} = 1: \mb{E} \lsrs{\abs{\inp{v}{X - \mu_S}}^{1 + \alpha}} \leq \frac{1}{2}.
    \end{equation*}
\end{lemma}

\begin{proof}
    We first note that:
    \begin{equation*}
        (\mu_S)_i = \begin{cases}
                        0, & \text{for } i \notin S \\
                        \frac{n^{- \frac{\alpha}{1 + \alpha}}\cdot d^{- \frac{(1 - \alpha)}{2(1 + \alpha)}}}{4}, & \text{otherwise}.
                    \end{cases}
    \end{equation*}
    Let $v$ be such that $\norm{v} = 1$. We have:
    \begin{align*}
        \mb{E} \lsrs{\abs{\inp{v}{X - \mu_S}}^{1 + \alpha}} &= \sum_{i \in S} \frac{1}{4n}\cdot \abs*{v_i\lprp{n^{\frac{1}{1 + \alpha}}\cdot d^{- \frac{(1 - \alpha)}{2(1 + \alpha)}} - \frac{n^{- \frac{\alpha}{1 + \alpha}}\cdot d^{- \frac{(1 - \alpha)}{2(1 + \alpha)}}}{4}}}^{1 + \alpha} \\
        & + \sum_{i \in S} \lprp{1 - \frac{1}{4n}} \cdot \abs*{v_i\lprp{\frac{n^{- \frac{\alpha}{1 + \alpha}}\cdot d^{- \frac{(1 - \alpha)}{2(1 + \alpha)}}}{4}}}^{1 + \alpha}.
    \end{align*}
    For the first term in this sum, we have:
    \begin{align*}
        &\sum_{i \in S} \frac{1}{4n}\cdot \abs*{v_i\lprp{n^{\frac{1}{1 + \alpha}}\cdot d^{- \frac{(1 - \alpha)}{2(1 + \alpha)}} - \frac{n^{- \frac{\alpha}{1 + \alpha}}\cdot d^{- \frac{(1 - \alpha)}{2(1 + \alpha)}}}{4}}}^{1 + \alpha} \leq \sum_{i \in S} \frac{1}{4n} \cdot \abs*{v_i n^{\frac{1}{1 + \alpha}}\cdot d^{- \frac{(1 - \alpha)}{2(1 + \alpha)}}}^{1 + \alpha} \\
        &\qquad = \frac{1}{4} \sum_{i \in S} \abs{v_i}^{1 + \alpha} \cdot d^{-\frac{(1 - \alpha)}{2}} \leq \frac{1}{4} \lprp{\sum_{i \in S} v_i^2}^{\frac{1 + \alpha}{2}} \cdot \lprp{\sum_{i \in S} d^{-1}}^{\frac{1 - \alpha}{2}} \leq \frac{1}{4},
    \end{align*}
    where the second inequality follows from H\"older's inequality. For the second term, we have:
    \begin{align*}
        \sum_{i \in S} \lprp{1 - \frac{1}{4n}} \cdot \abs*{v_i\lprp{\frac{n^{- \frac{\alpha}{1 + \alpha}}\cdot d^{- \frac{(1 - \alpha)}{2(1 + \alpha)}}}{4}}}^{1 + \alpha} &\leq \frac{1}{4}\sum_{i \in S} \abs{v_i}^{1 + \alpha} n^{-\alpha} d^{- \frac{(1 - \alpha)}{2}} \\
        &\leq \frac{1}{4}\sum_{i \in S} \abs{v_i}^{1 + \alpha} d^{- \frac{(1 - \alpha)}{2}} \leq \frac{1}{4},
    \end{align*}
    where the last inequality again follows from H\"older's inequality. Putting the two bounds together, we obtain:
    \begin{equation*}
        \forall v: \norm{v} = 1: \mb{E} \lsrs{\abs{\inp{v}{X - \mu_S}}^{1 + \alpha}} \leq \frac{1}{2}.
    \end{equation*}
\end{proof}

We now prove a lemma that establishes the optimality of Theorem~\ref{thm:ub_simple} in the regime of constant failure probability. We  use the following generative process for the data $\bm{X} = X_1, \dots, X_n$:

\begin{enumerate}
    \item Randomly pick a subset $S$ uniformly from the set $\{T \subset [d]: \abs{T} = d / 2\}$. 
    \item Generate $X_1, \dots, X_n$ iid from the distribution, $\mc{D}_S$.
\end{enumerate}

\begin{lemma}
    \label{lem:cplb}
    Let $(S, \bm{X})$ be generated according to the above process. We have, for any estimator $\hat{\mu} (\bm{X})$,
    \begin{equation*}
        \mb{P}_{S, \bm{X}} \lbrb{\norm{\hat{\mu} (\bm{X}) - \mu_S} \geq \frac{1}{24}\cdot \lprp{\frac{d}{n}}^{\frac{\alpha}{1 + \alpha}}} \geq \frac{1}{4}.
    \end{equation*}
\end{lemma}

\begin{proof}
    We first define the random variable $Y \coloneqq \sum_{i = 1}^n \bm{1} \lbrb{X_i \neq 0}$. From the definition of the distributions $\mc{D}_S$ we have:
    \begin{equation*}
        \mb{E} \lsrs{Y} = \frac{d}{8}
    \end{equation*}
    Therefore, we have that $Y \leq d / 4$ with probability at least $1 / 2$, by Markov's inequality. We now define the following random set: $T \coloneqq \{i \in [d]: \exists j \in [n]\text{ such that } (X_j)_i \neq 0\}$. We see from the definition of $T$ and $Y$ that $\abs{T} \leq Y$. We have with probability at least $1 / 2$ that $\abs{T} \leq d / 4$. Let $\bm{X}$ be an outcome for which $\abs{T} = k \leq d / 4$. We have by the symmetry of the distribution that:
    \begin{equation*}
        \mb{P} \lbrb{S | \bm{X}} =  \begin{cases}
                                        \frac{1}{\binom{d - k}{d/2 - k}}, &\text{if } T\subset S \text{ and } \abs{S} = d / 2\\
                                        0, &\text{otherwise}.
                                    \end{cases}
    \end{equation*}
    For given $\bm{X}$, define $Z_i = \bm{1} \lbrb{i \in S}$ for $i \notin T$ (For $i \in T$, $Z_i$ is $1$). We have for $Z_i$ and $Z_j$ for distinct $i, j \notin T$:
    \begin{equation*}
        \mb{E} \lsrs{Z_i | \bm{X}} = \mb{E} \lsrs{Z_j | \bm{X}} = \frac{d - 2k}{2(d - k)}.
    \end{equation*}
    Furthermore, we have:
    \begin{align*}
        \mathrm{Cov}(Z_i, Z_j | \bm{X}) &= \frac{(d - 2k)(d - 2k - 2)(d - k) - (d - 2k)^2 (d - k - 1)}{4(d - k)^2 (d - k - 1)} \\
        &= \frac{(d-2k)((d - 2k)(d - k) - 2(d - k) - (d - 2k)(d - k) + (d - 2k))}{4(d - k)^2 (d - k - 1)} \\
        &= \frac{-d (d - 2k)}{4(d - k)^2 (d - k - 1)} < 0.
    \end{align*}
    Now, consider some $R \subset [d]$ such that $\abs{R} = d / 2$ and $T \subset R$. Let $Q = R \setminus T$. For $Q$, we have $\abs{Q} = d / 2 - k$. We have for $S$:
    \begin{equation*}
        \abs{S \cap R} = k + \sum_{i \in Q} Z_i.
    \end{equation*}
    This means that:
    \begin{equation*}
        \mathrm{Var}\lprp{\abs{S \cap R} \mid \bm{X}} = \mathrm{Var}\lprp{\sum_{i \in Q} Z_i \mid \bm{X}} \leq \sum_{i \in Q} \lprp{\frac{d - 2k}{2(d - k)}}^2 \leq \frac{d / 2 - k}{4} = \frac{d}{8}.
    \end{equation*}
    Furthermore, we have that:
    \begin{equation*}
        \mb{E} \lprp{\abs{S \cap R} \mid \bm{X}} = k + \lprp{\frac{d}{2} - k}\cdot \frac{(d - 2k)}{2(d - k)} \leq \frac{d}{4} + \frac{d}{4} \cdot \frac{d}{4 (3d / 4)} = \frac{d}{4} + \frac{d}{12} = \frac{d}{3}.
    \end{equation*}
    Therefore, we have by Chebyshev's inequality that:
    \begin{equation*}
        \mb{P} \lbrb{\abs{S \cap R} \geq \frac{5d}{12}} \leq \frac{1}{2}.
    \end{equation*}

    Note that for any $S_1, S_2$ such that $\abs{S_i} = \frac{d}{2}$ and $\abs{S_1 \cap S_2} \leq \frac{5d}{12}$, we have:
    \begin{equation*}
        \norm{\mu_{S_1} - \mu_{S_2}} \geq \sqrt{2 \cdot \frac{d}{12} \cdot \lprp{\frac{n^{-\frac{\alpha}{1 + \alpha}}\cdot d^{-\frac{1 - \alpha}{2(1 + \alpha)}}}{4}}^2} \geq \frac{1}{12} \cdot \lprp{\frac{d}{n}}^{\frac{\alpha}{1 + \alpha}}.
    \end{equation*}
    Consider any estimator $\hat{\mu}$. Suppose that there exists $R$ such that $T \subset R$, $\abs{R} = d / 2$ and $\norm{\hat{\mu}(\bm{X}) - \mu_R} \leq \frac{1}{24}\cdot \lprp{\frac{d}{n}}^{\frac{\alpha}{1 + \alpha}}$. Then, we have that by the triangle inequality that:
    \begin{equation*}
        \mb{P} \lbrb{\norm{\hat{\mu} (\bm{X}) - \mu_S} \geq \frac{1}{24}\cdot \lprp{\frac{d}{n}}^{\frac{\alpha}{1 + \alpha}}\mid \bm{X}} \geq \frac{1}{2}.
    \end{equation*}
    In the alternate case where $\norm{\hat{\mu}(\bm{X}) - \mu_R} \geq \frac{1}{24} \cdot \lprp{\frac{d}{n}}^{\frac{\alpha}{1 + \alpha}}$ for all such $R$, the same conclusion holds true trivially. From  these two cases, we obtain:
    \begin{equation*}
        \mb{P} \lbrb{\norm{\hat{\mu} (\bm{X}) - \mu_S} \geq \frac{1}{24}\cdot \lprp{\frac{d}{n}}^{\frac{\alpha}{1 + \alpha}}\mid \bm{X}} \geq \frac{1}{2}.
    \end{equation*}
    Since such an $\bm{X}$ occurs with probability at least $1 / 2$, we arrive at our result:
    \begin{equation*}
        \mb{P} \lbrb{\norm{\hat{\mu} (\bm{X}) - \mu_S} \geq \frac{1}{24}\cdot \lprp{\frac{d}{n}}^{\frac{\alpha}{1 + \alpha}}} \geq \frac{1}{4}.
    \end{equation*}
\end{proof}

As part of our proof, we use the following one-dimensional lower bound from \cite{devroye2016sub}.

\begin{theorem}
    \label{thm:odlb}
    For any $n$, $\delta \in \lprp{2^{-\frac{n}{4}}, \frac{1}{2}}$, there exists a set of distributions $\mc{G}$ such that any $\mc{D} \in \mc{G}$ satisfies the weak-moment condition for some $\alpha > 0$ such that for any estimator $\hat{\mu}$:
    \begin{equation*}
        \mb{P}_{\mc{D} \in \mc{G}} \lbrb{\abs{\hat{\mu}(\bm{X}) - \mu(\mc{D})} \geq \lprp{\frac{\log 2 / \delta}{n}}^{\frac{\alpha}{1 + \alpha}}} \geq \delta
    \end{equation*}
    where $\bm{X}$ are drawn iid from $\mc{D}$.
\end{theorem}

Finally, we have the main theorem of the section:

\begin{theorem}
    \label{thm:lb}
    Let $n > 0$ and $\delta \in \lprp{e^{-\frac{n}{4}}, \frac{1}{4}}$. Then, there exists a set of distributions over $\mb{R}^d$, $\mc{F}$ such that each $\mc{D} \in \mc{F}$ satisfies the weak-moment condition for some $\alpha > 0$ and the following holds, for any estimator $\hat{\mu}$:
    \begin{equation*}
        \mb{P}_{\mc{D} \in \mc{F}} \lbrb{\norm{\hat{\mu} (\bm{X}) - \mu(\mc{D})} \geq \frac{1}{24}\cdot\max \lprp{\lprp{\frac{d}{n}}^{\frac{\alpha}{1 + \alpha}}, \sqrt{\frac{d}{n}}, \lprp{\frac{\log 2 / \delta}{n}}^{\frac{\alpha}{1 + \alpha}}}} \geq \delta,
    \end{equation*}
    where $\bm{X}$ are generated iid from $\mc{D}$. 
\end{theorem}

\begin{proof}[Proof of \cref{thm:lb}]
    The proof of theorem follows from Lemma~\ref{lem:cplb} by setting $\alpha$ to $1$ and $\alpha$ and from Theorem~\ref{thm:odlb}. 
\end{proof}

\section*{Acknowledgments}
We gratefully acknowledge the support of the NSF through grants IIS-1619362 and DMS-2023505 as well as support from the Mathematical Data Science program of the Office of Naval Research under grant number N00014-18-1-2764. We would like to thank Ohad Shamir for pointing out typographical errors in an earlier version of our draft.

\bibliographystyle{alpha}
\bibliography{refs}

\appendix

\section{Auxiliary Results}
\label{sec:aux}

\begin{lemma}
    \label{lem:mtmon}
    For any $\bm{Z} \in \mb{R}^{k \times d}$ and $x \in \mb{R}^d$, the optimal value of $\mt(x,r,\bm{Z})$ is monotonically nonincreasing in $r$.
\end{lemma}
\begin{proof}
The lemma follows trivially from the fact that a feasible solution $X$ of $\mt(x,r,\bm{Z})$ is also a feasible solution for $\mt(x,r^\prime,\bm{Z})$ for $r^\prime \leq r$.
\end{proof}

\begin{lemma}
    \label{lem:gexp}
    For $X \thicksim \ns(0, 1)$, $\mb{E} [\abs{X}] = \sqrt{\frac{2}{\pi}}$.
\end{lemma}
\begin{proof}
    \begin{align*}
        \mb{E} [\abs{X}] &= \int_{-\infty}^{\infty} \abs{x} \frac{1}{\sqrt{2\pi}} \exp\lbrb{-\frac{x^2}{2}} dx = 2\int_{0}^\infty x \frac{1}{\sqrt{2\pi}} \exp\lbrb{- \frac{x^2}{2}} dx \\
        &= \sqrt{\frac{2}{\pi}} \int_0^\infty \exp\{-t\} dt = \sqrt{\frac{2}{\pi}}.
    \end{align*}
\end{proof}

\begin{proof}[Proof of \cref{lem:lmom}]   
    The argument hinges on a Gaussian projection trick which introduces $g \sim \mathcal{N}(0, I)$ to rewrite the norm. From Lemma~\ref{lem:gexp} and the convexity of the function $f(x) = \abs{x}^{1 + \alpha}$, we have:
    \begin{align*}
        \mb{E} [\norm{X}^{1 + \alpha}] &= \mb{E}_X\lsrs{\lprp{\sqrt{\frac{\pi}{2}} \mb{E}_g \abs{\inp{X}{g}}}^{1 + \alpha}} \leq \frac{\pi}{2} \mb{E}_X \mb{E}_g \lsrs{\abs{\inp{X}{g}}^{1 + \alpha}} \\
        &= \frac{\pi}{2} \mb{E}_g \norm{g}^{1 + \alpha} \mb{E}_X \lsrs{\abs*{\inp*{X}{\frac{g}{\norm{g}}}}^{1 + \alpha}} \leq \frac{\pi}{2} \mb{E}_g [\norm{g}^{1 + \alpha}] \leq \frac{\pi}{2}\cdot d^{\frac{1 + \alpha}{2}}.
    \end{align*}
\end{proof}

The following result derives an analogue of the Chebyshev inequality that applies under the weak-moment assumption. The primary technical difficulty to showing concentrations of sums of such random variables is that we cannot exploit orthogonality of independent random variables in $L_2$ by ``expanding'' out the square---since the requisite second moments do not necessarily exist. 

\begin{proof}[Proof of \cref{lem:srv}]
The case where $\alpha = 0$ is trivial. When $\alpha > 0$, we start by defining:
\begin{equation*}
    S_i = \sum_{j = 1}^i X_j, \qquad S_0 = 0, \qquad f(x) = \abs{x}^{1 + \alpha}.
\end{equation*}
Therefore, we have from an application of Lemma~\ref{lem:fgbnd}:
\begin{align*}
    \mb{E} \lsrs{f(S_n)} &= \mb{E} \lsrs{\sum_{i = 1}^n f(S_i) - f(S_{i - 1})} = \sum_{i = 1}^n \mb{E} \lsrs{f(S_i) - f(S_{i - 1})} \\
    &= \sum_{i = 1}^n \mb{E} \lsrs{\int_{S_{i - 1}}^{S_i} f^\prime (x) dx} = \sum_{i = 1}^n \mb{E} \lsrs{X_i f^\prime (S_{i - 1}) + \int_{S_{i - 1}}^{S_i} f^\prime (x) - f^\prime (S_{i - 1})dx} \\
    &= \sum_{i = 1}^n \mb{E} \lsrs{\int_{S_{i - 1}}^{S_i} f^\prime (x) - f^\prime (S_{i - 1}) dx} \leq 2^{1 - \alpha}\sum_{i = 1}^n \mb{E} \lsrs{\int_{0}^{\abs{X_i}} f^\prime \lprp{\frac{t}{2}} dt} \\
    &= 2^{1 - \alpha}\sum_{i = 1}^n \mb{E} \lsrs{\int_{0}^{\abs{X_i} / 2} 2 f^\prime \lprp{s} ds} = 2^{2 - \alpha}\sum_{i = 1}^n \mb{E} \lsrs{f\lprp{\frac{\abs{X_i}}{2}}} \leq 2n.
\end{align*}
\end{proof}

\begin{lemma}
    \label{lem:fgbnd}
    Let $g(x) = \sgn(x) \abs{x}^{\alpha}$ for some $0 < \alpha \leq 1$. Then we have for any $h \geq 0$:
    \begin{equation*}
        \max_x g(x + h) - g(x) = 2^{1 - \alpha} h^\alpha.
    \end{equation*}
\end{lemma}
\begin{proof}
    Consider the function $l(x) = g(x + h) - g(x)$. We see that $l$ is differentiable everywhere except at $x = 0$ and $x = -h$. As long as $x \neq 0, -h$, we have:
    \begin{equation*}
        l^\prime (x) = g^\prime(x + h) - g^\prime(x) = \alpha (\abs{x + h}^{\alpha - 1} - \abs{x}^{\alpha - 1})
    \end{equation*}
    Since, we have $\alpha \leq 1$, $x = -\frac{h}{2}$ is a local maxima for $l(x)$. Furthermore, note that $l^\prime(x) \geq 0$ for $x \in (-\infty, -\frac{h}{2}) \setminus \{-h\}$ and $l^\prime (x) \leq 0$ for $x \in (-\frac{h}{2}, \infty) \setminus \{0\}$. Therefore, we get from the continuity of $l$ that $x = -\frac{h}{2}$ is a global maxima for $l(x)$.
\end{proof}



We now provide an auxiliary result which will be useful to controlling the moments of the thresholded versions of the vectors $X_i$.
\begin{lemma}
    \label{lem:mnshft}
    Let $\nu$ be a mean-zero distribution over $\mb{R}^d$ such that $X \ts \nu$ satisfies the weak-moment condition for some $\alpha > 0$. Furthermore, let $A \subset \mb{R}^d$ be such that $\nu(A) = \delta \leq \frac{1}{2}$. Let $\nu_S()$ be the conditional distribution of $\nu$ conditioned on the event $\{X \in S\}$ for any $X \subset \mb{R}^d$. Then we have for $Y \thicksim \nu(A^c)$:
    \begin{equation*}
        \text{Claim 1: } \norm{\mu(\nu_{A^c})} \leq 2\delta^{\frac{\alpha}{1 + \alpha}}, \qquad \text{Claim 2: }\forall v \in \mc{S}^{d - 1}, \quad \mb{E} \lsrs{\abs{\inp{v}{Y - \mu(\nu_{A^c})}}^{1 + \alpha}} \leq 20.
    \end{equation*}
\end{lemma}

\begin{proof}
    Letting $p_A = \mb{P}\lbrb{X \in A}$, we have $\nu = p_A \nu_A + p_{A^c} \nu_{A^c}$. Then,
    \begin{equation*}
        \norm{\mu(\nu) - \mu(\nu_{A^c})} = \max_{v \in \mc{S}^{d-1}} \inp{v}{\mu(\nu) - \mu(\nu_{A^c})}.
    \end{equation*}
    So for any $v \in \mc{S}^{d - 1}$:
    \begin{align*}
        \inp{v}{\mu(\nu) - \mu(\nu_{A^c})} &= \inp{v}{p_A \mu(\nu_A) + p_{A^c} \mu(\nu_{A^c}) - \mu(\nu_{A^c})} \\
        &= \inp{v}{p_A \mu(\nu_A) - p_{A} \mu(\nu_{A^c})} = p_A \inp{v}{\mu(\nu_A) - \mu(\nu_{A^c})}. 
    \end{align*}
    Since $\mu(\nu) = 0$, we have $p_A \mu(\nu_A) = -p_{A^c} \mu(\nu_{A^c})$. We now get:
    \begin{equation*}
        p_A\inp{v}{\mu(\nu_A) - \mu(\nu_{A^c})} = p_A\inp*{v}{\mu(\nu_A) + \frac{p_{A}}{p_{A^c}} \mu(\nu_{A})} = \lprp{1 + \frac{p_A}{p_{A^c}}} \inp{v}{p_A\mu(\nu_{A})}.
    \end{equation*}
    Finally,
    \begin{align*}
        \inp{v}{p_A \mu(\nu_{A})} &= \mb{E}_{X \ts \mu} \lsrs{\bm{1} \lbrb{X \in A} \inp{X}{v}} \\
        &\leq \lprp{\mb{E} \lsrs{(\bm{1} \lbrb{X \in A})^{\frac{1 + \alpha}{\alpha}}}}^{\frac{\alpha}{1 + \alpha}} \cdot \lprp{\mb{E} \lsrs{\abs{\inp{X}{v}}^{1 + \alpha}}}^{\frac{1}{1 + \alpha}} = p_{A}^{\frac{\alpha}{1 + \alpha}}
    \end{align*}
    where the inequality follows by an application of H\"older's inequality. We get the first claim as:
    \begin{equation*}
        \max_{v \in \mc{S}^{d-1}} \inp{v}{\mu(\nu) - \mu(\nu_{A^c})} = \lprp{1 + \frac{p_A}{p_{A^c}}} \inp{v}{p_A\mu(\nu_{A^c})} \leq \lprp{1 + \frac{p_A}{p_{A^c}}} p_A^{\frac{\alpha}{1 + \alpha}} \leq 2\delta^{\frac{\alpha}{1 + \alpha}},
    \end{equation*}
    where the final inequality follows from the fact that $p_{A^c} \geq p_{A}$. 
    
    For the second claim, let $Y \ts \nu_{A^c}$ and $\mu_Y = \mb{E}[Y]$. We decompose the required term as follows:
    \begin{equation*}
        \mb{E} \lsrs{\abs*{\inp{Y - \mu_Y}{v}}^{1 + \alpha}} \leq 2^{1 + \alpha} \cdot \mb{E} \lsrs{\abs*{\inp{\mu_Y}{v}}^{1 + \alpha} + \abs*{\inp{Y}{v}}^{1 + \alpha}}.
    \end{equation*}
    For the first term, we have with $Z \ts \nu_{A}$:
    \begin{equation*}
        \mb{E} \lsrs{\abs{\inp{Y}{v}}^{1 + \alpha}} = p_{A^c}^{-1} \lprp{\mb{E} \lsrs{\abs{\inp{X}{v}}^{1 + \alpha}} - p_A \mb{E} \lsrs{\abs{\inp{Z}{v}}^{1 + \alpha}}} \leq 2.
    \end{equation*}
    Therefore, we finally have:
    \begin{equation*}
        \mb{E} \lsrs{\abs{\inp{Y - \mu_Y}{v}}^{1 + \alpha}} \leq 8 + 2^{1 + \alpha}\cdot 2^{1 + \alpha}\cdot \delta^{\alpha}  \leq 16,
    \end{equation*}
    which proves the second claim of the lemma.
\end{proof}

\begin{lemma}
    \label{lem:trsig}
    Let $X \ts \nu$ be a mean-zero random vector satisfying the weak-moment condition for some $0 \leq \alpha \leq 1$. Then, we have for any $\tau > 0$:
    \begin{equation*}
        \mb{E} \lsrs{\norm{X}^2\cdot\bm{1} \lbrb{\norm{X} \leq \tau}} \leq \frac{\pi}{2} d^{\frac{1 + \alpha}{2}} \tau^{1 - \alpha}.
    \end{equation*}
\end{lemma}

\begin{proof}
    The proof of the lemma proceeds as follows:
    \begin{equation*}
        \mb{E} \lsrs{\norm{X}^2 \cdot \bm{1} \lbrb{\norm{X} \leq \tau}} \leq \tau^{1 - \alpha} \mb{E} \lsrs{\norm{X}^{1 + \alpha} \bm{1} \lbrb{\norm{X} \leq \tau}} \leq \tau^{1 - \alpha} \mb{E} \lsrs{\norm{X}^{1 + \alpha}} \leq \frac{\pi}{2} d^{\frac{1 + \alpha}{2}} \tau^{1 - \alpha},
    \end{equation*}
    where the last inequality follows from Lemma~\ref{lem:lmom}.

\end{proof}
\section{Initial Estimate}
\label{sec:init}

In this subsection, we analyze the initial estimate used in the thresholding step. We will show that the estimate is within $O(\sqrt{d})$ of the true mean with high probability.

\begin{lemma}
    \label{lem:init}
    Let $X_1, \dots, X_n$ be i.i.d. random vectors with mean $\mu$, satisfying the weak-moment condition for some $\alpha > 0$. Then the mean estimate, $\hat{\mu}$, provided by Algorithm~\ref{alg:inite} satisfies:
    \begin{equation*}
        \norm{\mu - \hat{\mu}} \leq 24\sqrt{d},
    \end{equation*}
    with probability at least $1 - e^{-\frac{n}{50}}$.
\end{lemma}

\begin{proof}
    Since our algorithm is translation invariant, we may assume without loss of generality that $\mu = \bm{0}$. Therefore, it suffices to prove that with probability at least $2^{-\Omega(n)}$:
    \begin{equation*}
        \norm{\hat{\mu}} \leq 16\sqrt{d}.
    \end{equation*}
    We have from Lemma~\ref{lem:lmom} that $\mb{E} [\norm{X}] \leq \frac{\pi}{2}\cdot \sqrt{d}$. Applying Markov's inequality:
    \begin{equation*}
        \mc{P} \lbrb{\norm{X} \leq 8\sqrt{d}}  \geq \frac{3}{4}.
    \end{equation*}
    Combining with Hoeffding's inequality we conclude that:
    \begin{equation*}
        \mc{P} \lbrb{\sum_{i = 1}^n \bm{1} \lbrb{\norm{X_i} \leq 8\sqrt{d}} \leq 0.6n} \leq \exp\lbrb{-\frac{n}{50}}.
    \end{equation*}

    Since Algorithm~\ref{alg:inite} returns as an estimate one of the data points, let $\hat{\mu} = X_i$ for some $i$ and let ${r_j = \min \{r > 0: \sum_{k = 1}^n \bm{1} \lbrb{\norm{X_j - X_k} \leq r} \geq 0.6n\}}$ for any $j \in [n]$. We now condition on the following event:
    \begin{equation*}
        \sum_{i = 1}^n \bm{1} \lbrb{\norm{X_i} \leq 8\sqrt{d}}  > 0.6n.
    \end{equation*}
    Let $\mc{S} = \lbrb{j: \norm{X_j} \leq 8\sqrt{d}}$. By the triangle inequality, for any $j \in \mc{S}$ we have:
    \begin{equation*}
        \sum_{k = 1}^n \bm{1} \lbrb{\norm{X_k - X_j} \leq 16\sqrt{d}} \geq 0.6n.
    \end{equation*}
    Therefore, by the definition of $\hat{\mu}$ we infer that $r_i \leq 16 \sqrt{d}$. Now, let $\mc{T} = \lbrb{k: \norm{X_k - X_i} \leq r_i}$. We have by the definition of $r_i$ that $\abs{\mc{T}} \geq 0.6n$. By the pigeonhole principle, we have that $\mc{T} \cap \mc{S} \neq \phi$. Let $j \in \mc{T} \cap \mc{S}$. By the triangle inequality we obtain:
    \begin{equation*}
        \norm{X_i} \leq \norm{X_i - X_j} + \norm{X_j} \leq 16\sqrt{d} + 8\sqrt{d} = 24\sqrt{d}.
    \end{equation*}
    Since the event being conditioned on occurs with probability at least $1 - e^{-\frac{n}{50}}$, this concludes the proof of the lemma.
\end{proof}
\section{Analyzing Relaxation}
\label{sec:ept}

We first show that the optimal value of the semidefinite program \ref{eq:mt} satisfies a bounded-difference condition with respect to the $Z_i$'s.

\begin{lemma}
  \label{lem:conc}
  Let $\bm{Y} = (Y_1, \dots, Y_k)$ be any set of $k$ vectors in $\mb{R}^d$. Now, let $\bm{Y}^\prime = (Y_1, \dots, Y_i^\prime, \dots, Y_k)$ be the same set of $k$ vectors with the $i^{th}$ vector replaced by $Y_i^\prime\in \mb{R}^d$. If $m$ and $m^\prime$ are the optimal values of $\mt(x,r,\bm{Y})$ and $\mt(x,r,\bm{Y}^\prime)$, we have:
  \begin{equation*}
    \abs{m - m^\prime} \leq 1.
  \end{equation*}
\end{lemma}

\begin{proof}
  First, assume that $X$ is a feasible solution to $\mt(x,r,\bm{Y})$. Let us define $X^\prime$ as:
  \begin{equation*}
    X^\prime_{i,j} = \begin{cases}
                      X_{i,j} & \text{ if $i,j \neq b_i$} \\
                      0 &\text{otherwise}.
                     \end{cases}
  \end{equation*}
  That is, $X^\prime$ is equal to $X$ except with the row and column corresponding to $b_i$ being set to $0$. We see that $X^\prime$ forms a feasible solution to $\mt(x,r,\bm{Y}^\prime)$. Therefore, we have that:
  \begin{equation*}
    \sum_{j = 1}^k X_{b_j, b_j} = \sum_{j = 1, j\neq i}^k X^\prime_{b_j, b_j} + X_{b_i, b_i} \leq \sum_{j = 1, j\neq i}^k X^\prime_{b_j, b_j} + 1 \leq m^\prime + 1,
  \end{equation*}
where the bound $X_{b_i,b_i} \leq 1$ follows from the fact that the $2\times 2$ submatrix of $X$ formed by the rows and columns indexed by $1$ and $b_i$ is positive semidefinite and from the constraint that $X_{b_i, b_i} = X_{1, b_i}$. Since the series of equalities holds for all feasible solutions $X$ of $\mt(x,r,\bm{Y})$, we get:
  \begin{equation*}
    m \leq m^\prime + 1.
  \end{equation*}
  Through a similar argument, we also conclude that $m^\prime \leq m + 1$. Putting the two inequalities together, we obtain the required conclusion.
\end{proof}
For the next few lemmas, we are concerned with the case where $x = \mu$. Since we already know that the optimal SDP value satisfies the bounded differences condition, we need to verify that the expectation is small. As a first step towards this, we define the 2-to-1 norm of a matrix $M$.

\begin{definition}
  The 2-to-1 norm of $M \in \mb{R}^{n \times d}$ is defined as
  \begin{equation*}
    \normtto{M} = \max_{\substack{\norm{v} = 1 \\ \sigma_i \in \{\pm 1\}}} \sigma^\top M v = \max_{\norm{v} = 1} \norm{Mv}_1.
  \end{equation*}
\end{definition}
We consider the classical semidefinite programming relaxation of the 2-to-1 norm. To start with, we will define a matrix $X \in \mb{R}^{(n + d + 1) \times (n + d + 1)}$ with the rows and columns indexed by $1$ and the elements $\sigma_i$ and $v_j$. The semidefinite programming relaxation is defined as follows:
\begin{gather*}
  \max \sum_{i, j} M_{i,j} X_{\sigma_i, v_j} \\
  \text{s.t } X_{1,1} = 1,\ \sum_{j = 1}^d X_{v_j,v_j} = 1,\ X_{\sigma_i,\sigma_i} = 1,\ X \succcurlyeq 0. \tag{TOR} \label{eq:tor}
\end{gather*}
We now state a theorem of Nesterov as stated in \cite{hopkins2018sub}:
\begin{theorem}{(\cite{nesterov1998semidefinite})}
  \label{thm:contr}
  There is a constant $K_{2\rightarrow 1} = \sqrt{\pi/2} \leq 2$ such that the optimal value, $m$, of the semidefinite programming relaxation~\ref{eq:tor} satisfies:
  \begin{equation*}
    m \leq K_{2\rightarrow 1} \normtto{M}.
  \end{equation*}
\end{theorem}
In the next step, we bound the expected 2-to-1 norm of the random matrix $Z$. To do this, we begin by recalling the Ledoux-Talagrand Contraction Theorem \cite{ledoux1991probability}.
\begin{theorem}
  \label{thm:ledtal}
  Let $X_1, \dots, X_n \in \mb{R}^d$ be i.i.d.~random vectors, $\mathcal{F}$ be a class of real-valued functions on $\mb{R}^d$ and $\sigma_i, \dots, \sigma_n$ be independent Rademacher random variables. If $\phi: \mb{R} \rightarrow \mb{R}$ is an $L$-Lipschitz function with $\phi(0) = 0$, then:
  \begin{equation*}
    \mathbb{E} \sup_{f \in \mathcal{F}} \sum_{i = 1}^n \sigma_i \phi(f(X_i)) \leq L\cdot \mb{E} \sup_{f \in \mathcal{F}} \sum_{i = 1}^{n} \sigma_i f(X_i).
  \end{equation*}
\end{theorem}

We are now ready to bound the expected 2-to-1 norm of the random matrix $Z$.
\begin{lemma}
  \label{lem:tto}
  Let $\bm{Y} = (Y_1, \dots, Y_n) \in \mb{R}^{n\times d}$ be a set of $n$ i.i.d.~random vectors such that $\mb{E} [Y_i] = 0$ and $\mb{E} [Y_i Y_i^\top] = \Lambda$ and assume that:
  \begin{equation*}
      \max_{v \in \mc{S}^{d - 1}} \mb{E} \lsrs{\abs*{\inp{v}{Y}}^{1 + \alpha}} \leq \beta.
  \end{equation*}
  Then we have:
  \begin{equation*}
    \mb{E} \normtto{\bm{Y}} \leq 2 \sqrt{n\Tr \Lambda} + n \beta^{\frac{1}{1 + \alpha}}.
  \end{equation*}
\end{lemma}
\begin{proof}
  Denoting by $Y$ and $Y_i^\prime$ random vectors that are independently and identically distributed as $Y_i$ and by $\sigma_i$ independent Rademacher random variables, we have: 
  \begin{align*}
    \mb{E} [\normtto{\bm{Y}}] &= \mb{E} \lsrs{\max_{\norm{v} = 1}\sum_{i = 1}^n \abs{\inp{Y_i}{v}}} = \mb{E} \lsrs{\max_{\norm{v} = 1}\sum_{i = 1}^n \abs{\inp{Y_i}{v}} + \mb{E} \abs{\inp{v}{Y_i}} - \mb{E} \abs{\inp{v}{Y_i}}}\\
    &\leq \mb{E} \lsrs{\max_{\norm{v} = 1} \sum_{i = 1}^n \abs{\inp{Y_i}{v}} - \mb{E} \abs{\inp{Y_i^\prime}{v}}} + n \max_{\norm{v} = 1}\mb{E} [\abs{\inp{v}{Y}}] \\
    &\leq \mb{E} \lsrs{\max_{\norm{v} = 1} \sum_{i = 1}^n \sigma_i (\abs{\inp{Y_i}{v}} - \abs{\inp{Y_i^\prime}{v}})} + n \max_{\norm{v} = 1} \mb{E} \lsrs{\abs{\inp{v}{Y}}}.
  \end{align*}
  Now, we have for the second term:
  \begin{equation*}
    \max_{\norm{v} = 1}\mb{E} [\abs{\inp{v}{Y}}] \leq \max_{\norm{v} = 1} \lprp{\mb{E} \inp{v}{Y}^{1 + \alpha}}^{\frac{1}{1 + \alpha}} \leq \beta^{\frac{1}{1 + \alpha}}.
  \end{equation*}
  For the first term, we employ a standard symmetrization argument:
  \begin{align*}
    \mb{E} \lsrs{\max_{\norm{v} = 1} \sum_{i = 1}^n \sigma_i (\abs{\inp{Y_i}{v}} - \abs{\inp{Y_i^\prime}{v}})} & \leq \mb{E} \lsrs{\max_{\norm{v} = 1} \sum_{i = 1}^n \sigma_i \abs{\inp{Y_i}{v}}} + \mb{E}\lsrs{\max_{\norm{v} = 1}\sum_{i = 1}^n-\sigma_i\abs{\inp{Y_i^\prime}{v}}} \\
    &= 2 \mb{E} \lsrs{\max_{\norm{v} = 1} \sum_{i = 1}^n \sigma_i\abs{\inp{v}{Y_i}}} \leq 2 \mb{E} \lsrs{\max_{\norm{v} = 1} \sum_{i = 1}^n \sigma_i\inp{v}{Y_i}} \\
    &= 2 \mb{E} \lsrs{\norm*{\sum_{i = 1}^n \sigma_i Y_i}} \leq 2 \lprp{\mb{E} \lsrs{\norm*{\sum_{i = 1}^n \sigma_i Y_i}^2}}^{1/2} \\
    &= 2 \lprp{\mb{E} \sum_{1 \leq i,j \leq n} \sigma_i \sigma_j \inp{Y_i}{Y_j}}^{1/2} = 2\sqrt{n\Tr \Lambda},
  \end{align*}
  where the second inequality follows from the Ledoux-Talagrand Contraction Principle (Theorem~\ref{thm:ledtal}).
  By putting these two bounds together, we prove the lemma.
\end{proof}

We now bound the expected value of $\mt(\mu,r,\bm{Y})$ by relating it to $\normtto{\bm{Y}}$.
\begin{lemma}
  \label{lem:exp}
  Let $\bm{Y} = (Y_1, \dots, Y_k) \in \mb{R}^{k\times d}$ be a collection of $k$ i.i.d.~random vectors with mean $\mu$ and covariance $\Lambda$ and assume that:
  \begin{equation*}
      \max_{v \in \mc{S}^{d - 1}} \mb{E} \lsrs{\abs*{\inp{v}{Y}}^{1 + \alpha}} \leq \beta.
  \end{equation*} 
  Denoting by $\mathcal{S}$ the set of feasible solutions for $\mt(\mu,r,\bm{Y})$, we have:
  \begin{equation*}
    \mb{E} \max_{x \in \mathcal{S}} \sum_{i = 1}^k X_{1, b_i} \leq \frac{1}{2r} \lprp{5 \sqrt{k \Tr \Lambda} + 2k \beta^{\frac{1}{1 + \alpha}}}.
  \end{equation*}
\end{lemma}

\begin{proof}
  First, let $X$ be a feasible solution for $\mt(\mu, r, \bm{Y})$. We construct a new matrix $W$ which is indexed by $\sigma_i$ and $v_j$ as opposed to $b_i$ and $v_j$ for $X$:
  \begin{gather*}
    W_{\sigma_i, \sigma_j} = 4X_{b_i, b_j} - 2 X_{1, b_i} - 2 X_{1, b_j} + 1,\quad W_{v_i, v_j} = X_{v_i, v_j},\quad W_{1,1} = 1, \\
    W_{1, v_i} = X_{1, v_i}, \quad W_{1, \sigma_i} = 2X_{1, b_i} - 1, \quad W_{v_i, \sigma_j} = 2X_{v_i, b_j} - X_{1, v_i}.
  \end{gather*}
  We prove that $W$ is a feasible solution to the SDP relaxation~\ref{eq:tor} of $\bm{Y} - \mu$. We see that:
  \begin{equation*}
    W_{\sigma_i, \sigma_i} = 1 \text{ and } \sum_{i = 1}^{d} W_{v_i, v_i} = 1.
  \end{equation*}
  Thus, we simply need to verify that $W$ is positive semidefinite. Let $w \in \mb{R}^{k + d + 1}$ be indexed by $1$, $\sigma_i$ and $v_j$. We construct from $w$ a new vector $w^\prime$, indexed by $1$, $b_i$ and $v_j$ and defined as follows:
  \begin{equation*}
    w^\prime_{1} = w_{1} - \sum_{i = 1}^k w_{\sigma_i}, \quad w^\prime_{b_i} = 2 w_{\sigma_i}, \quad w^\prime_{v_j} = w_{v_j}.
  \end{equation*}
  With $w^\prime$ defined in this way, we have the following equality:
  \begin{equation*}
    w^\top W w = (w^\prime)^\top X w^\prime \geq 0.
  \end{equation*}
  Since the condition holds for all $w \in \mb{R}^{k + d + 1}$, we get that $W \succcurlyeq 0$. Therefore, we conclude that $W$ is a feasible solution to the SDP relaxation~\ref{eq:tor} of $\bm{Y} - \mu$.

  We bound the expected value of $\mt(\mu,r,\bm{Y})$ as follows, denoting by $v_{b_i}$ the vector $(X_{b_i, v_1}, \dots, X_{b_i, v_d})$ and by $v$ the vector $(X_{1, v_1}, \dots, X_{1, v_d})$:
  \begin{align*}
    \mb{E} \max_{X \in \mathcal{S}} \sum_{i = 1}^k X_{1, b_i} &= \mb{E} \max_{X \in \mathcal{S}} \sum_{i = 1}^k X_{b_i, b_i} \leq \frac{1}{r}\mb{E} \max_{X \in \mathcal{S}} \sum_{i = 1}^k \inp{v_{b_i}}{Y_i - \mu} \\
    &= \frac{1}{2r} \mb{E} \max_{X \in S}\Big[ \sum_{i = 1}^k \inp{2v_{b_i} - v}{Y_i - \mu} + \sum_{i = 1}^k \inp{v}{Y_i - \mu}\Big] \\
    &\leq \frac{1}{2r} \lprp{\mb{E} \max_{X \in S} \sum_{i = 1}^k \inp{2v_{b_i} - v}{Y_i - \mu} + \mb{E} \max_{X \in \mathcal{S}} \sum_{i = 1}^k \inp{v}{Y_i - \mu}}.
  \end{align*}
  From the fact that $X$ is positive semidefinite, and from the fact that the $2\times 2$ submatrix indexed by $v_i$ and $b_j$ is positive semidefinite, we obtain:
  \begin{equation*}
    X^2_{v_i,b_j} \leq X_{v_i, v_i} X_{b_j, b_j} \leq X_{v_i, v_i} \implies \norm{v_{b_j}}^2= \sum_{i = 1}^d X^2_{v_i, b_j}  \leq \sum_{i = 1}^d X_{v_i, v_i} = 1.
  \end{equation*}
  Therefore, we get for the second term in the above equation:
  \begin{equation*}
    \mb{E} \max_{X \in \mathcal{S}} \sum_{i = 1}^{k} \inp{v}{Y_i - \mu} \leq \mb{E} \norm*{\sum_{i = 1}^{k} Y_i - \mu} \leq \lprp{\mb{E} \norm*{\sum_{i = 1}^{k} Y_i - \mu}^2}^{1/2} = (k \Tr \Lambda)^{1/2}.
  \end{equation*}
  We bound the first term using the following series of inequalities where $W$ is constructed from $X$ as described above:
  \begin{align*}
    \mb{E} \max_{x \in \mathcal{S}} \sum_{i = 1}^k \inp{2v_{b_i} - v}{Y_i - \mu} &= \mb{E} \max_{x \in \mathcal{S}} \sum_{i = 1}^k \sum_{j = 1}^d (Y_i - \mu)_j W_{\sigma_i, v_j} = \mb{E} \max_{x \in \mathcal{S}} \sum_{i = 1}^k \sum_{j = 1}^d (\bm{Y}_{i, j} - \mu_j) W_{\sigma_i, v_j} \\
    &\leq 2 \mb{E} \normtto{\bm{Y} - \bm{1}\mu^\top} \leq 4 \sqrt{k \Tr \Lambda} + 2 k \beta^{\frac{1}{1 + \alpha}},
  \end{align*}
  where the first inequality follows from Theorem~\ref{thm:contr} and the second inequality follows from Lemma~\ref{lem:tto}. By combining these three inequalities, we get:
  \begin{equation*}
    \mb{E} \max_{x \in \mathcal{S}} \sum_{i = 1}^k X_{1, b_i} \leq \frac{1}{2r} \lprp{5 \sqrt{k \Tr \Lambda} + 2 k \beta^{\frac{1}{1 + \alpha}}}.
  \end{equation*}
\end{proof}  

Finally, we establish the main technical result of this section, \cref{lem:rst_bound}.

\begin{proof}[Proof of \cref{lem:rst_bound}]
  From \cref{lem:exp}, we see that:
  \begin{equation*}
    \mb{E} \max_{X \in \mathcal{S}} \sum_{i = 1}^k X_{b_i, b_i} \leq \frac{k}{40}.
  \end{equation*}
  Now from \cref{lem:conc} and an application of the bounded difference inequality (see, for example, Theorem~6.2 in \cite{boucheron2013concentration}), with probability at least $1 - e^{k / 800}$:
  \begin{equation*}
    \max_{X \in \mathcal{S}} \sum_{i = 1}^k X_{b_i, b_i} \leq \frac{k}{20}.
  \end{equation*}
\end{proof}

In the following lemma, we analyze the set of random vectors returned by \cref{alg:part}. It will be useful to condition on the conclusion of \cref{lem:init}.

\begin{lemma}
  \label{lem:part}
  Let $\bm{X} = \{X_i\}_{i = 1}^n \ts \nu$ be iid zero-mean random vectors, satisfying the weak-moment condition for some $\alpha > 0$. Furthermore, suppose $x^\dagger$ satisfies $\norm{x^\dagger} \leq 24\sqrt{d}$. Then, the set of vectors $\bm{Y}$ returned by \cref{alg:part} with input $\bm{X}$ and $x^\dagger$ are iid random vectors vectors with mean $\tld{\mu}$ and covariance $\tld{\Sigma}$ and satisfy:
  \begin{gather*}
    \text{Claim 1: } \mb{P} \lbrb{\abs{\bm{Y}} \geq \frac{3n}{4}} \geq 1 - e^{- \frac{n}{50}}, \qquad \qquad \text{Claim 2: } \norm{\tld{\mu}} \leq 2 \lprp{\frac{d}{n}}^{\frac{\alpha}{1 + \alpha}} \\
    \text{Claim 3: } \forall \norm{v} = 1:\mb{E} \lsrs{\abs*{\inp{Y_i - \tld{\mu}}{v}}^{1 + \alpha}} \leq 20,\qquad\text{Claim 4: } \tr \tld{\Sigma} \leq 750 \max \lprp{n^{\frac{1 - \alpha}{1 + \alpha}}d^{-\frac{\alpha}{(1 + \alpha)}}, d}.
  \end{gather*}
\end{lemma}

\begin{proof}
  First, consider the set $A = \{x: \norm{x - x^\dagger} \leq \thr\}$ as defined in \cref{alg:part}. Note from the definition of the set $A$ that $\{x: \norm{x} \leq 0.75\thr\} \subseteq A$. We have from Markov's inequality and \cref{lem:lmom}:
  \begin{equation*}
    \mb{P} \lbrb{X_i \in A} \geq 1 - \min \lprp{\frac{d}{n}, \frac{1}{25}}
  \end{equation*}
  Therefore, by an application of Hoeffding's inequality, using the definition of the set of points $Y_1, \dots, Y_m$, we have that, with probability at least $1 - e^{- \frac{n}{50}}$:
  \begin{equation*}
    \abs{\bm{Y}} \geq \frac{3n}{4}
  \end{equation*}
  This proves the first claim of the lemma. For the next two claims, note that conditioned on the random variable $\hat{\mu}$, each of the $Y_i$ are independent and identically distributed according to $\nu_{A}$. Again, we get from the bound on $\mb{P} \lbrb{X_i \in A}$ by an application of \cref{lem:mnshft}, the next two claims of the lemma:
  \begin{equation*}
    \text{Claim 2: } \norm{\tld{\mu}} \leq 2 \lprp{\frac{d}{n}}^{\frac{\alpha}{1 + \alpha}}, \qquad \text{Claim 3: }\forall \norm{v} = 1:\mb{E} \lsrs{\abs*{\inp{Y_i - \tld{\mu}}{v}}^{1 + \alpha}} \leq 20.
  \end{equation*}
  For the final claim, note that as $\norm{x^\dagger} \leq 24 \sqrt{d}$, we have $A \subseteq B \coloneqq \lbrb{x: \norm{x} \leq 1.25\tau}$. Therefore, we have by the property of the mean that:
  \begin{align*}
    \tr \tld{\Sigma} = \mb{E} \lsrs{\norm{Y_i - \tld{\mu}}^2} &\leq \mb{E} \lsrs{\norm{Y_i}^2} = \frac{1}{\nu(A)} \mb{E} \lsrs{\norm{X_j}^2 \bm{1} \{X_j \in A\}} \\ 
    &\leq 2 \mb{E} \lsrs{\norm{X_j}^2 \bm{1} \{X_j \in B\}} \leq 750 \max \lprp{n^{\frac{1 - \alpha}{1 + \alpha}}d^{-\frac{\alpha}{(1 + \alpha)}}, d},
  \end{align*}
  where the final inequality follows from \cref{lem:trsig} and the definition of $\thr$.
\end{proof}

The main result of this section is the following high probability guarantee on the set of points output by \cref{alg:loce}.
\begin{lemma}
\label{lem:mnConc}
Let $\delta > e^{n / 8000}$ and $\bm{X} = \{X_i\}_{i = 1}^n$ be iid random vectors with mean $\mu$, satisfying the weak-moment condition for some known $\alpha > 0$. Furthermore, suppose that $x^\dagger$ satisfies $\norm{x^\dagger - \mu} \leq 24 \sqrt{d}$. Let $\bm{Z} = \{Z_i\}_{i = 1}^k$ denote the set of vectors output by \cref{alg:loce} run with inputs $\bm{X}$, $x^\dagger$ and $\delta$. Then, there exists a point $\tld{\mu}$ such that for all $r \geq \rad$:
\begin{equation*}
  \norm{\tld{\mu} - \mu} \leq 2 \lprp{\frac{d}{n}}^{\frac{\alpha}{1 + \alpha}}\text{ and } \max_{X \in \mathcal{S}} \sum_{i = 1}^k X_{b_i, b_i} \leq \frac{k}{20},
\end{equation*}
with probability at least $1 - \delta / 2$ where $\mathcal{S}$ denotes the set of feasible solutions of $\ref{eq:mt}(\tld{\mu}, r, \bm{Z})$.
\end{lemma}

\begin{proof}
  Note that it is sufficient to prove the lemma in the specific case where $\mu = \bm{0}$. We may now assume each of the $Y_i$ are iid random variables satisfying the conclusions of \cref{lem:part}. Therefore, $Z_i$ are iid random vectors with mean $\tld{\mu}$ and  covariance $\tld{\Sigma}$ satisfying:
  \begin{equation*}
    \norm{\tld{\mu}} \leq 2 \lprp{\frac{d}{n}}^{\frac{\alpha}{1 + \alpha}}\qquad \tr \tld{\Sigma} \leq \frac{750 k\max \lprp{n^{\frac{1 - \alpha}{1 + \alpha}}d^{-\frac{\alpha}{(1 + \alpha)}}, d}}{n}.
  \end{equation*}
  Furthermore, we have by an application of \cref{lem:srv} that:
  \begin{equation*}
    \forall \norm{v} = 1: \mb{E} \lsrs{\abs*{\inp{v}{Z_i - \tld{\mu}}}^{1 + \alpha}} \leq 80 \lprp{\frac{k}{n}}^{\alpha}.
  \end{equation*}
  Finally, the conclusion of the lemma follows by an application of \cref{lem:rst_bound} and the bound on the probabilities follows from the bounds in \cref{lem:rst_bound,lem:part}.
\end{proof}
\section{Gradient Descent Step}
\label{sec:grad_desc}

In this section, we analyze the gradient descent step in \cref{alg:meste}. This part of our proof is essentially identical to prior work for the finite covariance setting and we repeat it here for the sake of completeness \cite{cfb}. Throughout the section, we will analyze the convergence of gradient descent to an \emph{arbitrary} point $\smean$. However, in the final application, we will pick $\smean$ to be close to $\mu$. As discussed previously, the recovery guarantees of the gradient descent procedure are determined by the parameter $r^*$ defined below:

\begin{definition}
    \label{def:relax}
    For the bucket means, $\bm{Z} = (Z_1, \dots, Z_k)$, and point $\smean$, let $r^*$ be defined as follows:
    \begin{equation*}
        r^* \coloneqq \min \lbrb{r > 0: \ref{eq:mt} (\smean, r, \bm{Z}) \leq \frac{k}{20}}.
    \end{equation*}
\end{definition}
We also make use of the following remark implied by \cref{def:relax} (the implication follows by picking integral solutions for $X_{b_i, b_i}$ and setting the submatrix of $X$ corresponding to $v$ to be rank one in the semidefinite program \ref{eq:mt}:
\begin{remark}
    \label{rem:a21}
    For the bucket means, $\bm{Z} = (Z_1, \dots, Z_k)$, we have:
    \begin{equation*}
        \forall v \in \mb{R}^{d}, \norm{v} = 1\;\Rightarrow \abs*{\{i: \inp{Z_i - \smean}{v} \geq \radst\}} \leq 0.05k
    \end{equation*}
\end{remark}

\subsection{Distance Estimation Step}
\label{sec:disest}
In this subsection, we analyze the distance estimation step from \cref{alg:deste}. We  show that an accurate estimate of the distance of the current point from $\smean$ can be found. We begin by stating a lemma that shows that a feasible solution for $\ref{eq:mt}(x,r,\bm{Z})$ can be converted to a feasible solution for $\ref{eq:mt}(\smean,\radst,\bm{Z})$ with a reduction in optimal value. 
\begin{lemma}
  \label{lem:conv}
    Let $X \in \mb{R}^{(k + d + 1) \times (k + d + 1)}$ be a positive semidefinite matrix, symbolically indexed by $1$ and the variables $b_i$ and $v_j$. Moreover, suppose that $X$ satisfies:
  \begin{equation*}
    X_{1,1} = 1, \quad X_{b_i, b_i} = X_{1, b_i}, \quad \sum_{j = 1}^d X_{v_j, v_j} = 1, \quad \sum_{i = 1}^k X_{b_i, b_i} \geq 0.9k.
  \end{equation*}
  Then, there is a set of at least $0.85k$ indices $\mathcal{T}$ such that for all $i \in \mathcal{T}$:
  \begin{equation*}
    \inp{Z_i - \smean}{v_{b_i}} < X_{b_i, b_i} \radst,
  \end{equation*}
  and a set of at least $k / 3$ indices $\mathcal{R}$ such that for all $j \in \mathcal{R}$, we have $X_{b_j, b_j} \geq 0.85$.
\end{lemma}
\begin{proof}
  We prove the lemma by contradiction. Firstly, note that $X$ is infeasible for $\ref{eq:mt}(\smean, \radst, \bm{Z})$ as the optimal value for $\ref{eq:mt}(\smean,\radst,\bm{Z})$ is less than $k / 20$ (\cref{def:relax}). Note that the only constraints of $\ref{eq:mt}(\smean,\radst,\bm{Z})$ that are violated by $X$ are constraints of the form:
  \begin{equation*}
    \inp{Z_i - \smean}{v_{b_i}} < X_{b_i, b_i} \radst.
  \end{equation*}
Now, let $\mathcal{T}$ denote the set of indices for which the above inequality is violated. We can convert $X$ to a feasible solution for $\ref{eq:mt}(\smean, \radst, \bm{Z})$ by setting to zero the rows and columns corresponding to the indices in $\mathcal{T}$. Let $X^\prime$ be the matrix obtained by the above operation. We have from \cref{def:relax}:
  \begin{equation*}
    0.05k \geq \sum_{i = 1}^k X^\prime_{b_i, b_i} = \sum_{i = 1}^k X_{b_i, b_i} - \sum_{i \in \mathcal{T}} X_{b_i, b_i} \geq 0.9k - \abs{\mathcal{T}},
  \end{equation*}
where the last inequality follows from the fact that $X_{b_i, b_i} \leq 1$. By rearranging the above inequality, we get the first claim of the lemma.

For the second claim, let $\mathcal{R}$ denote the set of indices $j$ satisfying $X_{b_j, b_j} \geq 0.85$. We have:
\begin{equation*}
    0.9k \leq \sum_{j = 1}^k X_{b_j, b_j} = \sum_{j \in \mathcal{R}} X_{b_j, b_j} + \sum_{j \notin \mathcal{R}} X_{b_j, b_j} \leq \abs{\mathcal{R}} + 0.85k - 0.85\abs{\mathcal{R}} \implies \frac{k}{3} \leq \abs{\mathcal{R}}.
  \end{equation*}
  This establishes the second claim of the lemma.
\end{proof}
The following  lemma shows that if the distance between $\smean$ and a point $x$ is small then the estimate returned by \cref{alg:deste} is also small.
\begin{lemma}
\label{lem:destg}
  Suppose a point $x\in\mb{R}^d$ satisfies $\norm{x - \smean} \leq \radAlg$. Then, \cref{alg:deste} returns a value $d^\prime$ satisfying
  \begin{equation*}
    d^\prime \leq \radAlgG.
  \end{equation*}
\end{lemma}
\begin{proof}
Let $r^\prime = \radAlgG$. Suppose that the optimal value of $\ref{eq:mt}(x,r^\prime,\bm{Z})$ is greater than $0.9k$ and let its optimal solution be $X$. Let $\mathcal{R}$ and $\mathcal{T}$ denote the two sets whose existence is guaranteed by \cref{lem:conv}. From, the cardinalities of $\mathcal{R}$ and $\mathcal{T}$, we see that their intersection is not empty. For $j \in \mathcal{R} \cap \mathcal{T}$, we have:
  \begin{equation*}
    0.85r^\prime \leq \inp{Z_j - x}{v_{b_j}} = \inp{Z_j - \smean}{v_{b_j}} + \inp{\smean - x}{v_{b_j}} < \radst + \norm{x - \smean},
  \end{equation*}
where the first inequality follows from the fact that $j \in \mathcal{R}$ and the fact that $X$ is feasible for $\ref{eq:mt}(x, r^\prime, \bm{Z})$ and the last inequality follows from the inclusion of $j$ in $\mathcal{T}$ and Cauchy-Schwarz.

By plugging in the bounds on $r^\prime$ and $r$, we get:
  \begin{equation*}
    \norm{x - \smean} > 20.25\radst.
  \end{equation*}
  This contradicts the assumption on $\norm{x - \smean}$ and concludes the proof of the lemma.
\end{proof}
The next  lemma shows that the distance between $\smean$ and a point $x$ can be accurately estimated as long as $x$ is sufficiently far from $\smean$.
\begin{lemma}
\label{lem:dest}
  Suppose a point $x$ satisfies $\tilde{d} = \norm{x - \smean} \geq \radAlg$. Then, \cref{alg:deste} returns a value $d^\prime$ satisfying:
  \begin{equation*}
    0.95\tilde{d} \leq d^\prime \leq 1.25\tilde{d}.
  \end{equation*}
\end{lemma}
\begin{proof}
  Let us define the direction $\dd$ to be the unit vector in the direction of $x - \smean$. From \cref{rem:a21}, the number of $Z_i$ satisfying $\inp{Z_i - \smean}{\dd} \geq \radst$ is less than $k / 20$. Therefore, we have that for at least $0.95k$ points:
  \begin{equation*}
    \inp{Z_i - x}{-\dd} = \inp{x - \smean + \smean - Z_i}{\dd} = \norm{x - \smean} - \radst \geq 0.95\tilde{d}.
  \end{equation*}
Along with the monotonicity of $\ref{eq:mt}(x,r,\bm{Z})$ in $r$ (\cref{lem:mtmon}), this implies the lower bound.

For the upper bound, we show that the optimal value of $\ref{eq:mt}(x,1.25\tilde{d},\bm{Z})$ is less than $0.9k$. For the sake of contradiction, suppose that this optimal value is greater than $0.9k$. Let $X$ be a feasible solution of $\ref{eq:mt}(x, 1.25\tilde{d}, \bm{Z})$ that achieves $0.9k$. Let $\mathcal{R}$ and $\mathcal{T}$ be the two sets whose existence is guaranteed by \cref{lem:conv} and $j$ be an element in their intersection. We have for $j$:
  \begin{align*}
    0.85 (1.25\tilde{d}) &\leq X_{b_j,b_j} 1.25\tilde{d}\leq \inp{Z_j - x}{v_{b_j}} = \inp{Z_j - \smean}{v_{b_j}} + \inp{\smean - x}{v_{b_j}}  \\
    &< X_{b_j, b_j} \radst + \norm{\smean \!-\! x} 
    \!=\! X_{b_j, b_j}\radst \!+\! \tilde{d},
  \end{align*}
  where the first inequality follows from the inclusion of $j$ in $\mathcal{R}$ and the last inequality follows from the inclusion of $j$ in $\mathcal{T}$ and Cauchy-Schwarz.
By rearranging the above inequality, we get:
  \begin{equation*}
    X_{b_j, b_j} > (1.0625\tilde{d} - \tilde{d})(\radst)^{-1} > 1,
  \end{equation*}
which is a contradiction. Therefore, we get from the monotonicity of $\ref{eq:mt}(x,r,\bm{Z})$ (see \cref{lem:mtmon}), that $d^\prime \leq 1.25\tilde{d}$ and this concludes the proof of the lemma.
\end{proof}

\subsection{Gradient Estimation Step}\label{ssec:gradest}
In this section, we analyze the gradient estimation step of the algorithm. We show that an approximate gradient can be found as long as the current point $x$ is not too close to $\smean$. 
The following lemma shows that \cref{alg:geste} produces a nontrivial estimate of the gradient.
\begin{lemma}
  \label{lem:gest}
  Suppose a point $x$ satisfies $\norm{x - \smean} \geq \radAlg$ and let $\dd$ be the unit vector along $\smean - x$. Then, \cref{alg:geste} returns a vector $g$ satisfying:
  \begin{equation*}
    \inp{g}{\dd} \geq \frac{1}{15}.
  \end{equation*}
\end{lemma}
\begin{proof}
In the running of \cref{alg:geste}, let $X$ denote the solution of $\ref{eq:mt}(x,d^*,\bm{Z})$. We begin by factorizing the solution $X$ as $UU^\top$, with the rows of $U$ denoted by $u_1$, $u_{b_1}, \dots, u_{b_k}$ and $u_{v_1}, \dots, u_{v_d}$. We also define the matrix $U_v \!=\! (u_{v_1}, \dots, u_{v_d})$ in $\mb{R}^{(k + d + 1) \times d}$. From the constraints in \ref{eq:mt}, we~have:
  \begin{equation*}
    X_{b_i, b_i} = \norm{u_{b_i}}^2 \leq 1 \implies \norm{u_{b_i}} \leq 1,\quad \sum_{j = 1}^d X_{v_j, v_j} = \sum_{j = 1}^d \norm{u_{v_j}}^2 = \norm{U_v}_F^2 = 1 \implies \norm{U_v}_F = 1.
  \end{equation*}
Let $\mathcal{R}$ and $\mathcal{T}$ denote the sets defined in \cref{lem:conv}. Let $j \in \mathcal{T} \cap \mathcal{R}$. By noting that $v_{b_j} = u_{b_j}^\top U_v$, we have:
  \begin{equation*}
     0.85d^* \leq \inp{Z_j - \smean}{v_{b_j}} + \inp{\smean - x}{v_{b_j}} \leq X_{b_j, b_j} \radst + u_{b_j}^\top U_v (\smean - x),
  \end{equation*}
where the first inequality follows from the inclusion of $j$ in $\mathcal{R}$ and the second from its inclusion in $\mathcal{T}$. By rearranging this equation and using our bound on $d^*$ from \cref{lem:dest}, we obtain:
  \begin{equation}\label{eq:aboveeq}
     0.80 \norm{\smean - x} \leq 0.85d^* \leq X_{b_j, b_j} \radst + u_{b_j}^\top U_v (\smean - x).
  \end{equation}
By rearranging \cref{eq:aboveeq}, using Cauchy-Schwarz, $\norm{u_{b_i}} \leq 1$ and the assumption on $\norm{x - \smean}$:
  \begin{equation*}
    \norm{U_v (\smean - x)} \geq u_{b_j}^\top U_v (\smean - x) \geq 0.75 \norm{\smean - x},
  \end{equation*}
which yields:
\begin{equation*}
    \norm{U_v \dd} \geq 0.75.
  \end{equation*}
  Now, we have:
  \begin{equation*}
    1 = \norm{U_v}_F^2 = \norm{U_v \proj_\dd}_F^2 + \norm{U_v \projp_\dd}_F^2 \geq \norm{U_v \projp_\dd}_F^2 + (0.75)^2 \implies \norm{U_v \projp_\dd}_F \leq 0.67.
  \end{equation*}
 Let $y$ be the top singular vector of $X_v$. Note that $X_v = U_v^\top U_v$ and $y$ is also the top right singular vector of $U_v$. We have that:
  \begin{equation*}
    0.75 \leq \norm{U_v y} \leq \norm{U_v\proj_\dd y} + \norm{U_v\projp_\dd y} \leq \norm{\proj_\dd y} + \norm{U_v\projp_\dd}_F \leq \norm{\proj_\dd y} + 0.67.
  \end{equation*}
 This means that we have:
  \begin{equation*}
    \abs{\inp{y}{\dd}} \geq \frac{1}{15}.
  \end{equation*}
Note that the algorithm returns either $y$ or $-y$.  Consider the case where $\inp{y}{\dd} > 0$. From \cref{rem:a21} (implied by \cref{def:relax}), we have for at least $0.95k$ points:
  \begin{equation*}
    \inp{Z_i - \smean}{y} \leq \radst.
  \end{equation*}
  Therefore, we have for $0.95k$ points:
  \begin{equation*}
    \inp{Z_i - x}{y} = \inp{Z_i - \smean}{y} + \inp{\smean - x}{y} \geq - \radst + \frac{\radAlg}{15} > 0.
  \end{equation*}
This means that in the case where $\inp{y}{\dd} > 0$, we return $y$ which satisfies $\inp{\smean - x}{y} > 0$. This implies the lemma in this case. The case where $\inp{y}{\dd} < 0$ is similar, with $-y$ used instead of $y$. This concludes the proof of the lemma.
\end{proof}

We now prove a lemma regarding the output of \cref{alg:gd}.

\begin{lemma}
    \label{lem:gd_output}
    Let $\bm{Z} = \{Z_i\}_{i = 1}^k$ be $k$ points in $\mb{R}^d$, $\smean \in \mb{R}^d$ and $r^*$ be as in \cref{def:relax}. Then, \cref{alg:gd}, with input $\bm{Z}$, initialization $x^\dagger$, and number of iterations $T \geq \cgd \log \frac{\norm{\smean - x^\dagger}}{\epsilon}$ satisfies:
    \begin{equation*}
        \norm{x^* - \smean} \leq \max\lbrb{\radRet, \epsilon}.
    \end{equation*}
\end{lemma}
\begin{proof}
    First, let $\mc{G} = \{x: \norm{x - \smean} \leq 20r^*\}$. We prove the lemma in two cases:
  \begin{enumerate}
    \item[] \textbf{Case 1: } None of the iterates $\xt{t}$ lie in $\mathcal{G}$. In this case, we have from \cref{lem:dest}:
    \begin{equation}
    \label{eqn:dbo}
        0.95 \norm{\xt{t} - \smean} \leq \dt{t} \leq 1.25 \norm{\xt{t} - \smean}.
    \end{equation}
    This yields:
    \begin{align*}
      \norm{\xt{t + 1} - \smean}^2 &= \norm{\xt{t} - \smean}^2 - 2\frac{\dt{t}}{20}\inp{\gt{t}}{\smean - \xt{t}} + \frac{\dt{t}^2}{400} \leq \norm{\xt{t} - \smean}^2 - \frac{\dt{t}\norm{\smean - \xt{t}}}{150} + \frac{\dt{t}^2}{400} \\
      &\leq \norm{\xt{t} - \smean}^2 - \dt{t} \lprp{\frac{\norm{\mu - \xt{t}}}{150} - \frac{\dt{t}}{400}} \leq \lprp{1 - \frac{1}{500}} \norm{\xt{t} - \smean}^2,
    \end{align*}
    where the first inequality follows from \cref{lem:gest} and the last inequality follows by substituting the lower bound on $\dt{t}$ in the first term and the upper bound on $\dt{t}$ in the second term (\cref{eqn:dbo}). By an iterated application of the above inequality, we have:
    \begin{equation*}
        \norm{x^* - \smean} \leq \frac{1}{0.95}\cdot d^*\leq \frac{1}{0.95} \cdot \dt{T+1} \leq \frac{\epsilon}{10},
    \end{equation*}
    which concludes the proof of the lemma in this case.
    \item[] \textbf{Case 2: } One of the iterates $\xt{t}$ falls into the set $\mathcal{G}$. If the algorithm returns an element from $\mathcal{G}$, the lemma is true from the definition of $\mc{G}$. Otherwise, from \cref{lem:destg},  we have for $\xt{t}\in \mathcal{G}$ that $\dt{t} \leq \radAlgG$. Therefore, we have at the conclusion of the algorithm a value $\optd \leq \radAlgG$ along with a returned vector $x^*$ lying outside $\mathcal{G}$. Thus, we have from \cref{lem:dest}:
    \begin{equation*}
      0.95\norm{x^* - \smean} \leq \radAlgG \implies \norm{x^* - \smean} \leq \radRet.
    \end{equation*}
  \end{enumerate}
  This concludes the proof of the lemma.
\end{proof}
\section{Proof of Theorem~\ref{thm:ub_simple}}
\label{sec:pf_main}

We assemble the results established in other sections to prove \cref{thm:ub_simple}. Let $x^\dagger$ denote the output of \cref{alg:inite} in the running of \cref{alg:meste}. Note that this is passed as input to \cref{alg:part,alg:loce}. We now define the event $\mc{E}$:
\begin{equation*}
    \mc{E} \coloneqq \{\norm{x^\dagger - \mu} \leq 24 \sqrt{d}\}.
\end{equation*}
From \cref{lem:init,lem:mnConc}, we may assume the conclusions of \cref{lem:mnConc}, an event which occurs with probability at least $1 - \delta$. Since $\norm{x^\dagger - \mu} \leq 24 \sqrt{d}$, the proof of the theorem follows from \cref{lem:gd_output}, our bound on the number of iterations $T$ and the final desired accuracy. 

\qed
\section{Lower Bound for Robust Estimation under Weak Moments}
\label{sec:rob_lb}

In this section, we establish a lower bound for robust mean estimation under weak moments. The lower bound will be a consequence of the following theorem:

\begin{theorem}
    \label{thm:rob_lb}
    Given $\rob, \alpha \in (0, 1)$, there exist two distributions $\mc{D}_1$ and $\mc{D}_2$ over $\mb{R}$ with means $\mu_1$ and $\mu_2$, respectively, satisfying:
    
    \begin{enumerate}
        \item $d_{TV}(\mc{D}_1, \mc{D}_2) \leq \frac{\rob}{4}$
        \item $\abs{\mu_1 - \mu_2} \geq \frac{1}{4} \cdot \rob^{\alpha / (1 + \alpha)}$
        \item $\mb{E}_{X \ts \mc{D}_1} [\abs{X - \mu_1}^{1 + \alpha}], \mb{E}_{X \ts \mc{D}_2} [\abs{X - \mu_2}^{1 + \alpha}] \leq 1$.
    \end{enumerate}
\end{theorem}
\begin{proof}
    We prove the theorem by explicit construction. Let $\mc{D}_1$ be a $\delta$-distribution on $0$: $\mb{P}_{X \ts \mc{D}_1} (X = 0) = 1$. We have $\mu_1 = 0$ and the weak moment condition holds trivially for $\mc{D}_1$. Now, for $\mc{D}_2$, we have:
    
    \begin{equation*}
        \mb{P}_{X \ts \mc{D}_2} (X = x) = \begin{cases}
                                            1 - \frac{\rob}{4}, & \text{when } x = 0 \\
                                            \frac{\rob}{4}, & \text{when } x = \lprp{\frac{1}{\rob}}^{1 / (1 + \alpha)} \\
                                            0, &\text{otherwise}.
                                          \end{cases}.
    \end{equation*}
    
    From the definitions of $\mc{D}_1$ and $\mc{D}_2$, we obtain the first conclusion of the lemma. By direct computation, we have $\mu_1 = 0$ and $\mu_2 = \frac{1}{4} \cdot \rob^{\alpha / (1 + \alpha)}$. Finally, we verify the weak moment condition on $\mc{D}_2$ using the convexity of the function $f(x) = \abs{x}^{1 + \alpha}$:
    
    \begin{equation*}
        \mb{E}_{X \ts \mc{D}_2} [\abs{X - \mu_2}^{1 + \alpha}] \leq 2^{\alpha} \cdot \mb{E} [\abs{X}^{1 + \alpha} + \abs{\mu_2}^{1 + \alpha}] \leq 2^{\alpha} \lprp{\frac{1}{4} + \frac{\rob^\alpha}{4^{(1 + \alpha)}}} \leq 1.
    \end{equation*}
    
    This concludes the proof of the theorem.
\end{proof}

\end{document}